\documentclass[12pt]{amsart}
\usepackage{enumerate,amssymb,version,aliascnt,geometry,stmaryrd,mathrsfs,chngcntr}
\usepackage[all]{xy}
\usepackage{hyperref}

\hypersetup{
pdfauthor={Olivier Haution},
pdftitle={Involutions and Chern numbers of varieties},
pdfkeywords={involutions, fixed points, Chern numbers, algebraic cobordism},
hidelinks
}

\usepackage[T1]{fontenc}
\usepackage{textcomp}

\swapnumbers

\DeclareMathOperator{\id}{id}
\DeclareMathOperator{\Spec}{Spec}
\DeclareMathOperator{\Proj}{Proj}
\DeclareMathOperator{\Sym}{Sym}
\DeclareMathOperator{\im}{im}
\DeclareMathOperator{\coker}{coker}
\DeclareMathOperator{\carac}{char}
\DeclareMathOperator{\CH}{CH}
\DeclareMathOperator{\Ch}{Ch}
\DeclareMathOperator{\Cht}{\underline{Ch}}
\DeclareMathOperator{\CHt}{\underline{CH}}
\DeclareMathOperator{\CHx}{E}
\DeclareMathOperator{\CHa}{A}
\DeclareMathOperator{\Hh}{H}
\DeclareMathOperator{\Ht}{\underline{H}}
\DeclareMathOperator{\Rh}{R}

\DeclareMathOperator{\Sm}{\mathbf{Sm}}
\DeclareMathOperator{\op}{op}
\DeclareMathOperator{\Res}{Res}

\newcommand{\Oc}{\mathcal{O}}
\newcommand{\Ec}{\mathcal{E}}
\newcommand{\Ic}{\mathcal{I}}
\newcommand{\Sc}{\mathcal{S}}
\newcommand{\Lc}{\mathcal{L}}
\newcommand{\Ac}{\mathcal{A}}
\newcommand{\Zz}{\mathbb{Z}}
\newcommand{\Nn}{\mathbb{N}}
\newcommand{\Qq}{\mathbb{Q}}
\newcommand{\Laz}{\mathbb{L}}
\newcommand{\Ld}{\mathbb{L}_2}
\newcommand{\Lp}{\mathbb{L}_p}
\newcommand{\bb}{\mathbf{b}}
\newcommand{\Fd}{{\mathbb{F}_2}}
\newcommand{\Fp}{{\mathbb{F}_p}}
\newcommand{\Pp}{\mathbb{P}}
\newcommand{\Tan}{T}
\newcommand{\mud}{{\mu_2}}
\newcommand{\lc}{\llbracket}
\newcommand{\rc}{\rrbracket}
\newcommand{\Dec}[1]{\operatorname{Dec}({#1})}
\newcommand{\fund}[1]{{#1}_f}

\newtheorem{theorem}{Theorem}[section]
\newaliascnt{proposition}{theorem}
\newtheorem{proposition}[proposition]{Proposition}
\newaliascnt{lemma}{theorem}
\newtheorem{lemma}[lemma]{Lemma}
\newaliascnt{corollary}{theorem}
\newtheorem{corollary}[corollary]{Corollary}

\theoremstyle{definition}
\newaliascnt{remark}{theorem}
\newtheorem{remark}[theorem]{Remark}
\newaliascnt{example}{theorem}

\newaliascnt{definition}{theorem}
\newtheorem{definition}[definition]{Definition}
\newaliascnt{notation}{theorem}

\newtheoremstyle{par}                
  {}                                     
  {}                                     
  {}                                     
  {}                                     
  {}                 	                  
  {.}                                    
  { }                                    
  {}
\theoremstyle{par}
\newtheorem{para}[theorem]{}

\counterwithin{equation}{theorem}
\numberwithin{equation}{theorem}

\newcommand{\rref}[1]{(\ref{#1})}
\newcommand{\dref}[2]{(\ref{#1}.\ref{#2})}

\begin{document}

\begin{abstract}
Consider an involution of a smooth projective variety over a field of characteristic not two. We look at the relations between the variety and the fixed locus of the involution from the point of view of cobordism. We show in particular that the fixed locus has dimension larger than its codimension when certain Chern numbers of the variety are not divisible by two, or four. Some of those results, but not all, are analogues of theorems in algebraic topology obtained by Conner--Floyd and Boardman in the sixties. We include versions of our results concerning the vanishing loci of idempotent global derivations in characteristic two. Our approach to cobordism, following Merkurjev's \cite{Mer-Ori}, is elementary, in the sense that it does not involve resolution of singularities or homotopical methods.
\end{abstract}

\author{Olivier Haution}

\title{Involutions and Chern numbers of varieties}

\email{olivier.haution at gmail.com}
\address{Mathematisches Institut, Ludwig-Maximilians-Universit\"at M\"unchen, Theresienstr.\ 39, D-80333 M\"unchen, Germany}
\thanks{This work was supported by the DFG grant HA 7702/1-1 and Heisenberg fellowship HA 7702/4-1.}

\maketitle

\section{Introduction}
\numberwithin{theorem}{section}
\numberwithin{lemma}{section}
\numberwithin{proposition}{section}
\numberwithin{corollary}{section}
\numberwithin{example}{section}
\numberwithin{notation}{section}
\numberwithin{definition}{section}
\numberwithin{remark}{section}

Let $k$ be a field. The cobordism ring $\Laz$ is defined by identifying the smooth projective $k$-varieties which have the same collection of Chern numbers (indexed by monomials). Each such number is a geometrical invariant, defined as the degree of a monomial in the Chern classes of tangent bundle of the variety. Using instead modulo two Chern numbers yields the ring $\Ld$, a quotient of $\Laz$. Even though the base field $k$ is arbitrary, the ring $\Laz$ always coincides with the complex cobordism ring (the Lazard ring), and $\Ld$ with the unoriented cobordism ring. We will denote by $\lc Y \rc$ the class of a smooth projective $k$-variety $Y$ in either of these rings.

Consider an involution $\sigma$ of a connected smooth projective $k$-variety $X$. Assume that the characteristic of $k$ differs from two (that restriction may be lifted, see below). Denote by $N$ the normal bundle to the fixed locus $X^\sigma$ in $X$, by $\Pp(N)$ its projectivisation, and by $\Pp(N\oplus 1)$ its projective completion. Our first result is:
\begin{theorem}
\label{intro:main}
We have $\lc X\rc = \lc \Pp(N\oplus 1)\rc$ and $\lc \Pp(N)\rc=0$ in $\Ld$.
\end{theorem}

These equalities are just the first in a series of relations in $\Ld$, heuristically asserting that the map $\Pp(N \oplus 1) \to \Pp^{\infty}$ given by the canonical line bundle is bordant (in the unoriented sense) to any constant map $X \to \Pp^\infty$. A precise statement is given in \rref{th:X_PN_L2}. They imply the analogue of a formula due to Kosniowsky--Stong \cite{KS} in algebraic topology. Their formula is the basis of a vast collection of results concerning the fixed locus of smooth involutions of closed unoriented manifolds \cite{KS,Pergher-Stong} (see also for instance \cite{Kelton-I,Kelton-II,Pergher-Improvement,Pergher-Fn}), many of which could probably be translated into algebraic geometry. We refrain from doing so, but state the formula in  \rref{cor:KS}, and explain in detail how to derive it.\\

An example of Chern number is the so-called Euler number. Its value for a smooth projective $k$-variety $Y$ of pure dimension $n$ is the integer
\[
\chi(Y) = \deg c_n(\Tan_Y) = \sum_{i=0}^{2n} (-1)^i\dim_{\Qq_{\ell}} H^i_{et}(Y_{\overline{k}},\Qq_{\ell}),
\]
where $\ell$ is any prime number unequal to the characteristic of $k$, and $\overline{k}$ an algebraic closure of $k$. It is well-known that $\chi(X)$ and $\chi(X^\sigma)$ have the same parity, a fact which can be reproved using Theorem~\ref{intro:main}. The relations in $\Ld$ mentioned above imply the following analogue of a theorem of Conner--Floyd \cite[(27.4)]{CF-book-1st}:
\begin{theorem}[cf.\ \dref{cor:E_dim}{cor:E_dim:1}]
\label{intro:Euler_parity}
If $\chi(X)$ is odd, then $2 \dim X^\sigma \geq \dim X$.
\end{theorem}
Note that Theorem \ref{intro:Euler_parity} is only interesting when $\dim X$ is even, because the Euler number of an odd-dimensional variety is always even. In order to cover the odd-dimensional case, we really need to look beyond $\Ld$. This motivates the search for relations between $X$ and the normal bundle $N$ in a larger quotient of $\Laz$ than $\Ld$. Since the ideal $2\Laz \subset \Laz$ consists of classes of varieties admitting a fixed-point-free involution (exchanging two copies of a given variety), the largest quotient of $\Laz$ where the class of $X$ has any chance of being determined by $N$ is $\Laz/2$. We prove that this is indeed the case, by giving a formula expressing the class of $X$ in $\Laz/2$ in terms of the tautological line bundle $\Oc(-1) \to \Pp(N\oplus 1)$. As in Theorem \ref{intro:main}, this formula is part of a series of relations in $\Laz/2$, which are stated in \rref{th:Lmod2}. However, unlike Theorem \ref{intro:main}, this formula is not readily usable for computations, because it involves the formal group law. One may try to overcome this difficulty by focusing on one particular Chern number at a time.

This strategy works well in the case of the Euler number, allowing us to prove:
\begin{theorem}[cf.\ \dref{prop:E_fix}{prop:E_fix:2} and \dref{cor:E_dim}{cor:E_dim:2}]
\label{intro:Euler_4}
Assume that $\dim X$ is odd. Then
\begin{enumerate}[(i)]
\item $\chi(X) = \chi(X^\sigma) \mod 4$.

\item  If $\chi(X)$ is not divisible by $4$, then $2 \dim X^\sigma +1 \geq \dim X$.
\end{enumerate}
\end{theorem}
We could find in the literature no analogue of this theorem in algebraic topology. To prove it, the idea is to construct an oriented cohomology theory $\CHx$ which captures just enough information, in the sense that the class of a smooth projective $k$-variety in $\CHx(\Spec k)$ is determined by, and determines its dimension and Euler number. We then exploit the formula in $\Laz/2$ mentioned above by considering its trace in the theory $\CHx$.

An element of $\Laz/2$ or $\Ld$ is called decomposable if it is represented by a disjoint union of products of pairs of positive-dimensional smooth projective $k$-varieties. The decomposability of the class of a smooth projective $k$-variety is governed by the value of its so-called additive Chern number. We prove:
\begin{theorem}[cf.\ \rref{th:indec}]
\label{intro:dec}
Assume that the class of $X$ in $\Laz/2$ is indecomposable. Then $2 \dim X^\sigma +1 \geq \dim X$.
\end{theorem}
When $\dim X+1$ is not a power of two, decomposabilities in $\Laz/2$ and $\Ld$ are equivalent (see \dref{lemm:indec_Ld_L2}{lemm:indec_Ld_L2:2}), and Theorem~\ref{intro:dec} follows from the Theorem~\ref{intro:main} (and the corresponding supplementary relations in $\Ld$). In this case, Theorem~\ref{intro:dec} is an algebraic analogue of a theorem of Boardman in topology \cite[second part of Theorem 1]{Boardman-BAMS}. We are not aware of a topological analogue of Theorem~\ref{intro:dec} when $\dim X +1$ is a power of two. As above, the idea for the proof in that case is to construct an oriented cohomology theory $\CHa$ such that the class of a smooth projective $k$-variety in $\CHa(\Spec k)$ is determined by, and determines its dimension and additive Chern number.\\

All of our results are actually valid when the characteristic of $k$ is arbitrary, provided that we consider $\mud$-actions instead of involutions. In characteristic not two, those concepts coincide. A $\mud$-action in characteristic two is an idempotent global derivation, and the fixed locus $X^\mud$ is the vanishing locus of the derivation.

Of course involutions do exist in characteristic two, and it would be interesting to cover that case also. The category of smooth projective varieties, crucial for the use of cobordism theory, seems inadequate in that case, because the constant group $\Zz/2$ is not linearly reductive, and it is easy to find involutions on smooth projective varieties whose fixed locus is singular. If one is willing to work with singular schemes, it is possible to obtain results on involutions in characteristic two involving the Segre class of the normal cone to the fixed locus \cite{df}.\\

Finally, let us explain why we limit ourselves to $\mu_p$-actions for the prime $p=2$. If $\mu_p$ acts on a smooth $k$-variety $X$, any eigenvalue for the induced $\mu_p$-action on the normal bundle $N$ to the fixed locus in $X$ must be a nontrivial $p$-th root of unity. When $p=2$, the only possible eigenvalue is $-1$, so that $\mu_p$ must act trivially on $\Pp(N)$ (see \rref{lemm:inv_blowup} below). For odd primes $p$, results of the type given in this paper would necessarily involve the normal bundle $N$ together with its $\mu_p$-action, substantially reducing their usability. As an illustration, assume that $\dim X>0$ and that $k$ is algebraically closed. In case $p=2$, the relations in  $\Ld$ mentioned just after Theorem \ref{intro:main} imply that the number of fixed points, if finite, must be even (see \rref{cor:trivial_normal}). As explained in \cite{isol}, when $p$ is odd we can only say that this number cannot be one. The integer which must be prime to $p$ is the number of fixed points counted with multiplicities determined by the $\mu_p$-actions on the tangent spaces (see \cite[(4.3.4)]{isol} for a precise formula).

\newpage

\section{Oriented cohomology theories}
\numberwithin{theorem}{subsection}
\numberwithin{lemma}{subsection}
\numberwithin{proposition}{subsection}
\numberwithin{corollary}{subsection}
\numberwithin{example}{subsection}
\numberwithin{notation}{subsection}
\numberwithin{definition}{subsection}
\numberwithin{remark}{subsection}

\label{sect:oct}

\subsection{Axioms}
\label{sect:basic}
\hfill\\

\vspace*{-1em}

We fix a base field $k$ for the whole paper. We denote by $\Sm_k$ the category of smooth quasi-projective $k$-schemes. The tangent bundle of $X \in \Sm_k$ is denoted by $\Tan_X$, and the Grothendieck group of vector bundles on $X$ by $K_0(X)$.

\begin{para}
Let $E$ be a vector bundle over $X \in \Sm_k$, and $\Ec$ its $\Oc_X$-module of sections. We will denote by $\Pp(E)$ or $\Pp(\Ec)$ the scheme $\Proj_X(\Sym \Ec^\vee)$. This is the dual of the convention used in \cite{LM-Al-07}.
\end{para}

\begin{definition}
A cartesian square in $\Sm_k$
\begin{equation}
\label{eq:transverse}
\begin{gathered}
\xymatrix{
W\ar[r]^h \ar[d]_e & Z \ar[d]^g\\ 
Y \ar[r]^f& X
}
\end{gathered}
\end{equation}
is called \emph{transverse} if, for every connected component $W_0$ of $W$, denoting by 
$Y_0,Z_0,X_0$ the connected components of $Y,Z,X$ containing the images of $W_0$,
\[
\dim W_0 + \dim X_0 = \dim Y_0 + \dim Z_0.
\]
\end{definition}

\begin{definition}[{\cite[Definition~1.1.2]{LM-Al-07}}]
\label{def:oct}
A functor $\Hh$ from $\Sm_k^{\op}$ to the category of $\Zz$-graded rings, together with a group morphism $f_*^{\Hh} \colon \Hh(Y) \to \Hh(X)$ for each projective morphism $f\colon Y \to X$ in $\Sm_k$, is called an \emph{oriented cohomology theory} if the conditions \eqref{def:oct:1}--\eqref{def:oct:HI} below are satisfied. We write $f^*_{\Hh}$ instead of $\Hh(f)$ when $f$ is a morphism in $\Sm_k$, and denote by $\Hh^n(X)$ the degree $n$ component of $\Hh(X)$.

\begin{enumerate}[(i)]
\item \label{def:oct:1} If $X,Y \in \Sm_k$ are connected and $f\colon Y \to X$ is a projective morphism, then $f_*^{\Hh}$ is graded of degree $\dim X -\dim Y$.

\item For any $X,Y \in \Sm_k$ the natural morphism $\Hh(X \sqcup Y) \to \Hh(X) \times \Hh(Y)$ is bijective.

\item \label{def:oct:PF} If $f\colon Y \to X$ is a projective morphism in $\Sm_k$, then $f_*^{\Hh} (a f^*_{\Hh}(b)) = f_*^{\Hh}(a)b$ for any $a\in \Hh(Y), b \in \Hh(X)$.

\item If $X \in \Sm_k$, then $(\id_X)_* = \id_{\Hh(X)}$. If $f\colon Y \to X$ and $g \colon Z \to Y$ are projective morphisms in $\Sm_k$, then $f^{\Hh}_* \circ g^{\Hh}_* =(f \circ g)^{\Hh}_*$.

\item \label{def:oct:TR} Given a transverse square \eqref{eq:transverse} with $f$ projective, we have $h_*^{\Hh} \circ e^*_{\Hh} = g^*_{\Hh} \circ f_*^{\Hh}$.

\item \label{def:oct:PB} Let $E$ be a vector bundle over $X \in \Sm_k$ and $p\colon \Pp(E) \to X$ the associated projective bundle. Denote by $s\colon \Pp(E) \to \Oc_{\Pp(E)}(1)$ the zero-section of the canonical bundle, and write $\xi = s^*_{\Hh} \circ s_*^{\Hh}(1)$. Then $1,\xi,\cdots,\xi^{r-1}$ is a basis of the $\Hh(X)$-module $\Hh(\Pp(E))$ (for the structure induced by $p^*_{\Hh}$).

\item \label{def:oct:HI} Let $p\colon V \to X$ be a torsor under a vector bundle over $X \in \Sm_k$. Then $p^*_{\Hh}\colon \Hh(X) \to \Hh(V)$ is bijective.
\end{enumerate}
\end{definition}

\begin{para}
One sees easily that the axioms of \rref{def:oct} are equivalent to those of \cite[\S2]{Mer-Ori}, with the word ``tautological'' replaced by ``canonical'' in \cite[\S2, (iii)]{Mer-Ori}. Moreover it follows from \eqref{eq:-1} below that oriented cohomology theories in the sense of \rref{def:oct} are also oriented cohomology theories in the sense of \cite[\S2]{Mer-Ori}.
\end{para}

\begin{para}
\label{p:CH_is_oct}
The Chow ring $\CH$ is an oriented cohomology theory, see e.g.\ \cite{Ful-In-98}.
\end{para}

For the rest of \S \ref{sect:oct}, we fix an oriented cohomology theory $\Hh$.

\begin{para}
\label{ex:tensor_oct}
If $\Hh(\Spec k) \to R$ is a morphism of $\Zz$-graded rings, then the functor $\Hh \otimes_{\Hh(\Spec k)} R$ is naturally an oriented cohomology theory.
\end{para}

\begin{para}
\label{p:Chern}
Let $V$ be a vector bundle of rank $r$ over $S\in \Sm_k$. Using the notation of \dref{def:oct}{def:oct:PB} for $E=V^\vee$ and $X=S$, the Chern classes $c_i^{\Hh}(V) \in \Hh^i(S)$ are defined using Grothendieck's method \cite{Gro-58} by setting $c_i^{\Hh}(V) =0$ if $i \not \in \{0,\cdots,r\}$ and
\[
\sum_{i=0}^r (-1)^ip_{\Hh}^*\big(c_i^{\Hh}(V)\big) \xi^{r-i} = 0 \in \Hh(\Pp(E)).
\]
\end{para}
\begin{para}
We will use the simplified notation $f_*,f^*,c_i$ instead of $f_*^{\Hh},f^*_{\Hh}, c_i^{\Hh}$ when no confusion seems likely to arise. If $j\colon Y \to X$ is a closed immersion in $\Sm_k$, we will write $[Y] = j_*(1) \in \Hh(X)$. 
\end{para}

\begin{para}
\label{p:c1}
Let $L$ be a line bundle over $X\in \Sm_k$. Then $c_1(L) = s^* \circ s_*(1)$, where $s \colon X \to L$ is the zero-section (this follows from \dref{def:oct}{def:oct:PB}).
\end{para}

\begin{para}
\label{p:c1_Cartier}
If $D \to X$ is an effective Cartier divisor in $\Sm_k$, then $[D] = c_1(\Oc(D)) \in \Hh(X)$. This follows from \rref{p:c1}, \dref{def:oct}{def:oct:HI}, \dref{def:oct}{def:oct:TR} (see \cite[Proposition~3.2]{Mer-Ori}).
\end{para}

\begin{para}
If $f\colon Y \to X$ is a morphism in $\Sm_k$ and $E$ a vector bundle over $X$, then $f^*c_i(E) = c_i(f^*E) \in \Hh(Y)$ for all $i$.
\end{para}

\begin{para}
\label{p:obm}
An oriented cohomology theory defines an ``oriented Borel--Moore weak homology theory on $\mathbf{Sm}_k$'' by \cite[Proposition 5.2.4]{LM-Al-07}, hence an ``oriented Borel--Moore functor of geometric type on $\mathbf{Sm}_k$'' by \cite[Remark 4.1.10]{LM-Al-07}.
\end{para}

\begin{para}
\label{p:nilp}
Let $X \in \Sm_k$. Then for $n$ large enough, and line bundles $L_1,\cdots,L_n$ over $X$, we have $c_1(L_1)\cdots c_1(L_n) =0 \in \Hh(X)$. This follows from \rref{p:obm} and \cite[Lemma~4.1.3]{LM-Al-07}.
\end{para}

\begin{para}
\label{p:frac}
Let $L$ be a line bundle over $X \in \Sm_k$. By \rref{p:nilp}, we may evaluate a power series in $\Hh(\Spec k)[[x]]$ at $c_1(L)$ to obtain an element of $\Hh(X)$. If $f,g \in \Hh(\Spec k)[[x]]$ are such that $f=gh$ for some $h \in \Hh(\Spec k)[[x]]$, we will write
\[
\frac{f(c_1(L))}{g(c_1(L))} = h(c_1(L)) \in \Hh(X).
\]
\end{para}

\begin{para}
\label{p:Whitney}
(Whitney product formula) If $0\to E_1 \to E_2 \to E_3\to 0$ is an exact sequence of vector bundles over $X \in \Sm_k$, then for all $n$
\[
c_n(E_2) = \sum_{i=0}^n c_i(E_1) c_{n-i}(E_3) \in \Hh(X).
\]
This follows from \cite[Proposition~4.1.15]{LM-Al-07}, in view of \rref{p:obm}.
\end{para}

\begin{para}
\label{p:splitting}
(Splitting principle) Let $E$ be a vector bundle of rank $r$ over $X\in \Sm_k$. Then there is a composite of projective bundles $q\colon P \to X$ such that $q^*E$ admits a filtration by subbundles whose successive quotients are line bundles $L_1,\cdots,L_r$. By \dref{def:oct}{def:oct:PB} the pullback $q^* \colon \Hh(X) \to \Hh(P)$ is injective. By the Whitney product formula \rref{p:Whitney}, we have $c_i(q^*E)=\sigma_i(c_1(L_1),\cdots,c_1(L_r))$ for all $i=0,\cdots,r$, where $\sigma_i$ is the $i$-th elementary symmetric function in $r$ variables.
\end{para}

\begin{para}
\label{p:Chern_virtual}
Let $E$ be a vector bundle over $X \in \Sm_k$. Then the class $c_i(E)$ vanishes for $i$ large enough by \rref{p:nilp} and \rref{p:splitting}, so that the class $c(E) = 1+c_1(E) + \cdots$ is invertible in $\Hh(X)$. If $E_1,E_2$ are vector bundles over $X\in \Sm_k$, then the class $c(F) = c(E_1)c(E_2)^{-1} \in \Hh(X)$ depends only on $F=E_1-E_2 \in K_0(X)$ by the Whitney product formula \rref{p:Whitney}. We denote by $c_i(F)$ its component in $\Hh^i(X)$.
\end{para}

\begin{para}
Let $V$ be a vector bundle of rank $r$ over $S \in \Sm_k$, and $p\colon \Pp(V) \to S$ the associated projective bundle. We may view the tautological line bundle $\Oc(-1)$ as a subbundle of $p^*V$. The quotient $Q = p^*V/\Oc(-1)$ has rank $r-1$, hence the Whitney product formula \eqref{p:Whitney} yields in $\Hh(\Pp(V))$
\[
0=c_r(Q) = \sum_{i=0}^r c_i(p^*V) c_{r-i}(-\Oc(-1))  = \sum_{i=0}^r (-1)^{r-i}p^*\big(c_i(V)\big) c_1(\Oc(-1))^{r-i}.
\]
It follows that the Chern classes $c_i$ are the same as those defined in \cite[\S3]{Mer-Ori}.
\end{para}

\subsection{Cohomology of the point}

\begin{para}
When $X$ is a smooth projective $k$-scheme, with structural morphism $p \colon X \to \Spec k$, we will write in $\Hh(\Spec k)$
\[
\lc X \rc = p_*(1) \quad \text{ and } \quad \lc u \rc = p_*(u) \text{ for } u \in \Hh(X).
\]
When $\Hh=\CH$ (or $\CH/p$), we will write $\deg(u) \in \Zz$ (or $\Fp$), instead of $\lc u\rc$. 
\end{para}

\begin{para}
\label{p:Hc}
We will denote by $\fund{\Hh} \subset \Hh(\Spec k)$ the subgroup generated by the elements $\lc X \rc$, for $X$ a smooth projective $k$-scheme. It is a graded subring (see \cite[Proposition~2.5]{Mer-Ori}), whose degree $n$ component we will denote by $\fund{\Hh}^n$.
\end{para}

\begin{proposition}[{\cite[Corollary~9.10]{Mer-Ori}}]
Let $E_1,\cdots,E_n$ be vector bundles over $X \in \Sm_k$, and $i_1,\cdots,i_n \in \Nn$. Then $\lc c_{i_1}(E_1)\cdots c_{i_n}(E_n)\rc \in \fund{\Hh}$.
\end{proposition}

\subsection{The formal group law}

\begin{para}
A commutative formal group law is a pair $(R,F)$, where $R$ is a commutative ring and $F\in R[[x,y]]$ a power series satisfying
\begin{enumerate}[(i)]
\item $F(x,y) = F(y,x) \in R[[x,y]]$,
\item $F(x,0)=x \in R[[x]]$,
\item $F(x,F(y,z)) = F(F(x,y),z) \in R[[x,y,z]]$.
\end{enumerate}	
\end{para}

\begin{para}
By \cite[Lemma~1.1.3]{LM-Al-07}, there is a power series 
\begin{equation}
\label{eq:FGL}
F_{\Hh}(x,y)=x+_{\Hh}y = \sum_{i,j} a_{i,j} x^iy^j \in \Hh(\Spec k) [[x,y]]
\end{equation}
with $a_{i,j} \in \Hh^{1-i-j}(\Spec k)$ such that for any line bundles $L,M$ over $X \in \Sm_k$
\[
c_1^{\Hh}(L \otimes M) = c_1^{\Hh}(L) +_{\Hh} c_1^{\Hh}(M) \in \Hh(X),
\]
and the pair $(\Hh(\Spec k),F_{\Hh})$ is a commutative formal group law. The coefficients $a_{i,j}$ actually belong to $\fund{\Hh}$ by \cite[Remark~2.5.8]{LM-Al-07} and \rref{p:obm}. 
\end{para}

\begin{para}
\label{p:additive}
We say that the theory $\Hh$ is \emph{additive} if $F_{\Hh}(x,y)=x+y$. An example of additive theory is $\CH$, see e.g.\ \cite[Proposition~2.5 (e)]{Ful-In-98}.
\end{para}

\begin{para}
\label{p:formal_mult}
Let $[-1]_{\Hh}(x) \in \fund{\Hh}[[x]]$ be the unique power series such that $F_{\Hh}(x, [-1]_{\Hh}(x)) = x$. For $a\in \Zz$, we define a power series $[a]_{\Hh}(x) \in \fund{\Hh}[[x]]$, called \emph{formal multiplication by $a$}, by setting $[0]_{\Hh}(x)=0$, and iteratively $[a]_{\Hh}(x) = F_{\Hh}([a-1]_{\Hh}(x),x)$ for $a>0$, as well as $[a]_{\Hh}(x)=[-1]_{\Hh} ([-a]_{\Hh}(x))$ for $a<0$. The leading term of the power series $[a]_{\Hh}(x)$ is $ax$. In particular, there are elements $a_i \in \fund{\Hh}$ such that
\begin{equation}
\label{eq:-1}
[-1]_{\Hh}(x) = -x + \sum_{i\geq 1} a_i x^{i+1} \in \fund{\Hh}[[x]].
\end{equation}
\end{para}

\subsection{Deformation to the normal bundle}
\begin{lemma}
\label{lemm:deformation}
Let $T \to X$ be a closed immersion in $\Sm_k$, with normal bundle $N$. Denote by $Y$ the blowup of $T$ in $X$; its exceptional divisor is $\Pp(N)$. Write $\zeta = c_1(\Oc_{\Pp(N\oplus 1)}(-1)) \in \Hh(\Pp(N\oplus 1))$ and $\eta = c_1(\Oc_Y(\Pp(N))) \in \Hh(Y)$. Using the convention of \rref{p:frac}, for any $g \in \Hh(\Spec k)[[x]]$, we have in $\Hh(\Spec k)$
\[
\lc g(\eta) \rc = g(0)\lc X \rc + \Big\lc g(\zeta)\frac{[-1]_{\Hh}(\zeta)}{\zeta} \Big\rc.
\]
\end{lemma}
\begin{proof}
Denote by $B$ the blowup of $T\times 0$ in $X \times \Pp^1$, and by $j\colon \Pp(N\oplus 1) \to B$ the immersion of the exceptional divisor. Then $B$ naturally contains $Y$ and $X=X \times 1$ as closed subschemes. By \rref{p:obm} and \cite[Proposition~2.5.1]{LM-Al-07} (and the proof of \cite[Proposition~2.5.2]{LM-Al-07}), we have
\begin{equation}
\label{eq:def}
[Y] = [X] + j_*\frac{[-1]_{\Hh}(\zeta)}{\zeta}  \in \Hh(B).
\end{equation}
Let $\rho = c_1(\Oc_B(\Pp(N\oplus 1))) \in \Hh(B)$. Then $\rho$ restricts to $\zeta$ on $\Pp(N\oplus 1)$, to $\eta$ on $Y$, and to zero on $X$. The statement follows by multiplying \eqref{eq:def} with $g(\rho)$ and projecting to $\Hh(\Spec k)$.
\end{proof}

\subsection{Vishik's formula}
\hfill\\

\vspace*{-1em}
When $\Hh$ is the algebraic cobordism and $k$ has characteristic zero, the next statement is due to Vishik \cite[\S5.4]{Vi-Sym} (he mentions that similar computations were performed earlier independently by Rost and Smirnov). We reproduce Vishik's proof, with minor alterations required when $\carac k=2$.

\begin{proposition}
\label{lemm:Vishik}
Let $f \colon Y \to Z$ be a finite morphism in $\Sm_k$ whose fiber over any generic point of $Z$ is the spectrum of a two-dimensional algebra. Then the $\Oc_Z$-module $\Lc = \coker(\Oc_Z \to f_*\Oc_Y)$ is locally free of rank one, and we have
\[
f_*[Y] = \frac{[2]_{\Hh}(c_1(\Lc^\vee))}{c_1(\Lc^\vee)} \in \Hh(Z).
\]
\end{proposition}
\begin{proof}
The $\Oc_Z$-module $\Ac=f_*\Oc_Y$ is locally free of rank two, see e.g.\ \cite[\S 4, n$^{\circ}$ 5, cor. de la prop. 8]{Bou-AC-10}. The morphism of $\Ac$-modules $1\otimes \id \colon \Ac \to \Ac \otimes_{\Oc_Z} \Ac$ admits a retraction (the multiplication map of the $\Oc_Z$-algebra $\Ac$), and it follows that its cokernel $\Lc \otimes_{\Oc_Z} \Ac$ is a locally free $\Ac$-module of rank one. By faithful flatness of the $\Oc_Z$-algebra $\Ac$, the $\Oc_Z$-module $\Lc$ is locally free of rank one.

To prove the remaining statement, we may assume that $Z$ is affine by Jouanolou's trick \cite[Lemme~1.5]{Jouanolou} (in view of \dref{def:oct}{def:oct:TR} and \dref{def:oct}{def:oct:HI}). Let $\Sc^\bullet$ be the symmetric algebra on the $\Oc_Z$-module $\Ac$. Consider the morphisms of $\Oc_Z$-modules $\nu \colon \Ac \to \Sc^2$ and $\mu \colon \Sc^2 \to \Sc^2$ given by $\nu(a)= 1\otimes a$ and $\mu(a \otimes b)= a\otimes b - 1 \otimes ab$. Then $\ker \mu = \im \nu$, and $\coker \nu = \Lc^{\otimes 2}$. This gives an inclusion $\Lc^{\otimes 2}=\im \mu\subset \Sc^2$. The induced morphism of $\Nn$-graded $\Oc_{\Sc^\bullet}$-modules $\Lc^{\otimes 2} \otimes_{\Oc_Z} \Sc^{\bullet -2} \to \Sc^\bullet$ is injective, because locally the $\Oc_Z$-module $\Lc^{\otimes 2}$ is freely generated by a nonzero element of $\Sc^2$ and $\Sc^\bullet$ is an integral domain. Its image is the homogeneous ideal $\Ic \subset \Sc^\bullet$ generated by $\im \mu$. The morphism of $\Oc_Z$-modules $\Ac \to \Ac[t]$ given by $a \mapsto at$ induces a morphism of $\Nn$-graded $\Oc_Z$-algebras $\Sc^\bullet \to \Ac[t]$  whose kernel is $\Ic$, and which is surjective in degrees $\geq 1$. Thus the closed subscheme of $\Pp(\Ac^\vee)=\Proj_Z \Sc^\bullet$ defined by the homogeneous ideal $\Ic$ of $\Sc^\bullet$ is isomorphic to $\Proj_Z \Ac[t] \simeq \Spec_Z \Ac \simeq Y$ as a $Z$-scheme. We have realised the $Z$-scheme $Y$ as a Cartier divisor in $\Pp(\Ac^\vee)$ whose line bundle is $p^*\Lc^{\otimes -2}(2)$, where $p\colon \Pp(\Ac^\vee)\to Z$ is the projective bundle.

The sequence of $\Oc_Z$-modules $0 \to \Oc_Z \to \Ac \to \Lc \to 0$ splits, because $Z$ is affine. The corresponding inclusion $\Oc_Z \subset \Ac^\vee$ defines an effective Cartier divisor $j \colon \Pp(\Oc_Z) \to \Pp(\Ac^\vee)$ whose line bundle is $p^*\Lc^\vee(1)$. Since $\Oc_{\Pp(\Ac^\vee)}(1)$ has trivial restriction on $\Pp(\Oc_Z)$, in view of \rref{p:c1_Cartier} and \dref{def:oct}{def:oct:PF} we have in $\Hh(\Pp(\Ac^\vee))$
\[
[Y] = c_1(p^*\Lc^{\otimes -2}(2))=[2]_{\Hh}(c_1(p^*\Lc^\vee(1))) = j_*\Big(\frac{[2]_{\Hh}(c_1(j^*p^*\Lc^\vee))}{c_1(j^*p^*\Lc^\vee)} \Big).
\]
Since $p\circ j$ is an isomorphism, we conclude by applying $p_*$ and using \dref{def:oct}{def:oct:PF}.
\end{proof}

\section{The universal twisting}
\label{sect:cobordism}

\subsection{Twisting a theory}
\hfill\\

\vspace*{-1em}
In this section $\Hh$ is an oriented cohomology theory.

\begin{para}
\label{p:bb}
When $R$ is a $\Zz$-graded ring, we denote by $R[\bb]$ the polynomial ring over $R$ in the variables $b_i$ for $i \in \Nn -\{0\}$. The ring $R[\bb]$ is $\Zz$-graded by letting $b_i$ have degree $-i$. If $f\colon R \to S$ is a group morphism between $\Zz$-graded rings, we will again denote by $f$ the induced group morphism $R[\bb] \to S[\bb]$.
\end{para}

\begin{para}
\label{p:fund_pol}
Consider the power series (where $b_0 =1$)
\[
\pi(x) = \sum_{i \in \Nn} b_ix^i \in \Zz[\bb][[x]].
\]
If $L$ is a line bundle over $X \in \Sm_k$, then $\pi(c_1(L)) \in \Hh(X)[\bb]^{\times}$ by \rref{p:nilp}. It follows from the splitting principle \rref{p:splitting} that there is a unique way to define for every $X\in \Sm_k$ a map $P^{\Hh} \colon K_0(X) \to \Hh(X)[\bb]$ satisfying
\begin{enumerate}[(i)]
\item $f^*P^{\Hh}(E) = P^{\Hh}(f^*E)$ for any morphism $f\colon Y \to X$ in $\Sm_k$ and $E\in K_0(X)$,

\item $P^{\Hh}(L) = \pi(c_1(L))$ when $L$ is a line bundle over $X \in \Sm_k$,
\item $P^{\Hh}(E + F) = P^{\Hh}(E) P^{\Hh}(F)$ for any $X\in \Sm_k$ and $E,F \in K_0(X)$.
\end{enumerate}
\end{para}

\begin{para}
A sequence of integers $\alpha=(\alpha_1,\cdots,\alpha_m)$ with $m \in \Nn$ is called a \emph{partition} if $\alpha_1 \geq \alpha_2 \geq \cdots \geq \alpha_m>0$. We will write $|\alpha| = \alpha_1 + \cdots + \alpha_m \in \Nn$. To the partition $\alpha$ corresponds the monomial $b_\alpha = b_{\alpha_1} \cdots b_{\alpha_m} \in \Zz[\bb]$.
\end{para}

\begin{para}
\label{p:def_CF}
Let $X\in \Sm_k$ and $E\in K_0(X)$. Observe that $P^{\Hh}(E)$ has degree zero in the $\Zz$-graded ring $\Hh(X)[\bb]$. We define the \emph{Conner--Floyd Chern class} $c_\alpha^{\Hh}(E) \in \Hh^{|\alpha|}(X)$ (or simply $c_\alpha(E)$) for each partition $\alpha$ by the formula
\[
P^{\Hh}(E) = \sum_\alpha c^{\Hh}_\alpha(E) b_\alpha \in \Hh(X)[\bb].
\]
\end{para}

\begin{para}
\label{p:Chern_number}
When $X$ is a smooth projective $k$-scheme and $\alpha$ a partition, the corresponding \emph{Chern number} is
\[
c_\alpha(X) = \deg \big(c_\alpha^{\CH}(-\Tan_X)\big) \in \Zz.
\]
\end{para}

\begin{para}
\label{p:c_alpha_c_i}
When $\alpha$ is the partition $(1,\cdots,1)$ with $|\alpha|=n$, we have $b_\alpha = b_1^n$ and $c_\alpha(E) = c_n(E)$ for any $X \in \Sm_k$ and $E \in K_0(X)$.
\end{para}

\begin{para}
\label{p:CF-Q}
Let $\alpha=(\alpha_1,\cdots,\alpha_m)$ and let $n\geq m$. The $n$-th symmetric group acts on the ring $\Zz[x_1,\cdots,x_n]$ by permuting the variables. The sum of the elements in the orbit of $x_1^{\alpha_1} \cdots x_m^{\alpha_m}$ may be written as a polynomial $Q_\alpha$ in the elementary symmetric functions $\sigma_1,\cdots,\sigma_m$, which does not depend on the choice of $n$. 
\end{para}

\begin{para}
\label{p:Q-basis}
Any homogeneous polynomial of degree $d$ in $\Zz[y_1,\cdots,y_n]$, where $y_i$ has degree $i$, is a $\Zz$-linear combination of the polynomials $Q_\alpha(y_1,\cdots,y_n)$ for $|\alpha| =d$. 
\end{para}

\begin{lemma}
\label{lemm:CF-Q}
Let $\alpha$ be a partition. For any $X \in \Sm_k$ and $E \in K_0(X)$, we have
\[
c_\alpha(E) = Q_\alpha(c_1(E),c_2(E),\cdots) \in \Hh(X).
\]
\end{lemma}
\begin{proof}
This follows from the construction \rref{p:fund_pol} when $E$ is a vector bundle. In general, we may assume that $X$ is affine by Jouanolou's trick \cite[Lemme~1.5]{Jouanolou} (in view of \dref{def:oct}{def:oct:TR} and \dref{def:oct}{def:oct:HI}). Then there is an integer $s$ such that $E +s \in K_0(X)$ is the class of a vector bundle, and
\[
c_\alpha(E) = c_\alpha(E+s) =  Q_\alpha(c_1(E+s),c_2(E+s),\cdots) =  Q_\alpha(c_1(E),c_2(E),\cdots).\qedhere
\]
\end{proof}

\begin{lemma}
\label{lemm:c_alpha_inv}
Let $\alpha$ be a partition. Then there are elements $\lambda_{\alpha,\beta} \in \Zz$ for all partitions $\beta$ with $|\beta|=|\alpha|$, such that for any $X\in \Sm_k$ and $E\in K_0(X)$ we have
\[
c_\alpha(E) = \sum_{|\beta|=|\alpha|}  \lambda_{\alpha,\beta} c_\beta(-E) \in \Hh(X).
\]
\end{lemma}
\begin{proof}
We proceed by induction on $|\alpha|$, the case $\alpha=\varnothing$ being clear. From the relation $P^{\Hh}(E) P^{\Hh}(-E) =1$ we deduce, using the induction hypothesis 
\[
-c_\alpha(E) = \sum_{\substack{b_\gamma b_\delta= b_\alpha\\ \gamma \neq \alpha}} c_{\gamma}(E) c_{\delta}(-E) = \sum_{\substack{b_\gamma b_\delta= b_\alpha\\ \gamma \neq \alpha}} \sum_{|\varepsilon| = |\gamma|} \lambda_{\gamma,\varepsilon} c_{\varepsilon}(-E) c_{\delta}(-E).
\]
It follows from \rref{p:Q-basis} that $Q_{\varepsilon} Q_{\delta}$ is a $\Zz$-linear combination of the polynomials $Q_{\beta}$, for $|\beta|=|\varepsilon| + |\delta|$. Thus the statement follows from \rref{lemm:CF-Q}.
\end{proof}

\begin{para}
For $X \in \Sm_k$ we set $\Ht(X) = \Hh(X)[\bb]$ and for a morphism $f\colon Y \to X$ in $\Sm_k$ we set $f^*_{\Ht} = f^*_{\Hh}$ (we use the notation of \rref{p:bb}). If $f$ is projective with virtual tangent bundle $\Tan_f \in K_0(Y)$, for any $a\in \Ht(Y)$ we set
\[
f_*^{\Ht}(a)= f_*^{\Hh}(P^{\Hh}(-\Tan_f)a) \in \Ht(X).
\]
\end{para}

\begin{proposition}
\label{prop:tilde_is_oct}
The functor $\Ht$, together with the above defined pushforwards, is an oriented cohomology theory.
\end{proposition}
\begin{proof}
See \cite[\S 7.4.2]{LM-Al-07} or \cite[Proposition~4.3]{Mer-Ori}.
\end{proof}

\begin{para}
\label{p:def_exp}
We define the power series (where $b_0=1$)
\[
\exp(x) = x\pi(x) = \sum_{i\in \Nn} b_i x^{i+1} \in \Zz[\bb][[x]].
\]
\end{para}

\begin{para}
\label{p:c1_tilde}
If $L$ is a line bundle over $X \in \Sm_k$, then $c_1^{\Ht}(L) = \exp(c_1^{\Hh}(L)) \in \Ht(X)$. This follows from \rref{p:c1} and \dref{def:oct}{def:oct:PF} (see \cite[Lemma~4.2]{Mer-Ori}).
\end{para}

\begin{para}
\label{p:FGL_exp}
Denoting by $\exp^{-1}$ the composition inverse of $\exp$, we have in $\Ht(\Spec k)[[x,y]]$
\[
x+_{\Ht} y = \exp(\exp^{-1}(x) +_{\Hh} \exp^{-1}(y)).
\]
This follows from \rref{p:c1_tilde} (see \cite[Lemma~8.1]{Mer-Ori})
\end{para}

\subsection{The cobordism ring}

\begin{para}
We will denote by $\Laz$ the subring $\fund{\CHt} \subset \CHt(\Spec k) = \Zz[\bb]$ defined in \rref{p:Hc}. When $p$ is a prime number, we will write $\Ch = \CH/p$ and denote by $\Lp$ the subring $\fund{\Cht}\subset \Cht(\Spec k) = \Fp[\bb]$.
\end{para}

\begin{para}
\label{p:Chern_Laz}
Let $X$ be a smooth projective $k$-scheme. Then, using the notation of \rref{p:Chern_number}
\[
\lc X \rc = \sum_{\alpha} c_\alpha(X) b_\alpha \in \Laz.
\]
\end{para}

\begin{theorem}[{\cite[Theorem~8.2]{Mer-Ori}}]
\label{th:Lazard}
The pair $(\Laz,F_{\CHt})$ is the universal commutative formal group law.
\end{theorem}
\begin{corollary}
\label{cor:Lazard}
The ring $\Laz$ is generated by the coefficients $a_{i,j}$ of \eqref{eq:FGL}.
\end{corollary}
\begin{proof}
By construction \cite[II, \S5]{Adams-stable}, the coefficient ring of the universal commutative formal group law is generated by the coefficients of the corresponding power series. Thus the corollary follows from \rref{th:Lazard}.
\end{proof}

\begin{lemma}
\label{lemm:FGL_Lp}
We have $[p]_{\Cht}(x) = 0$.
\end{lemma}
\begin{proof}
Since $[p]_{\Ch}(x) = 0$ by \rref{p:additive}, this follows from \rref{p:c1_tilde}.
\end{proof}

\begin{proposition}
\label{prop:Lp_univ}
The kernel of the surjective morphism $\Laz \to \Lp$ is the ideal generated by the coefficients of the power series $[p]_{\CHt}(x) \in \Laz[[x]]$. Thus $(\Lp,F_{\Cht})$ is the universal commutative formal group law whose formal multiplication by $p$ (see \rref{p:formal_mult}) vanishes.
\end{proposition}
\begin{proof}
Let $(\Gamma,\Phi)$ be the universal commutative formal group law whose formal multiplication by $p$ vanishes. By \cite[Proposition 7.3]{Quillen-Elementary}, this law admits a logarithm, that is a power series $l \in \Gamma[[t]]$ with leading coefficient $t$ such that
\[
\Phi(x,y) = l^{-1}(l(x)+l(y)) \in \Gamma[[x,y]],
\]
where $l^{-1}$ denotes the composition inverse of $l$.

The morphism $\Laz \to \Gamma$ classifying the formal group law $(\Gamma,\Phi)$ is surjective, and its kernel is the ideal generated by the coefficients of the power series $[p]_{\CHt}(x) \in \Laz[[x]]$. By \rref{lemm:FGL_Lp}, the surjective morphism $\Laz \to \Lp$ factors through a surjective morphism $\Gamma \to \Lp$. To conclude the proof, we will provide a retraction of the composite $\pi \colon \Gamma \to \Lp \subset \Fp[\bb]$. Consider the morphism $\varphi \colon \Fp[\bb] \to \Gamma$ sending $b_i$, for $i\geq 1$, to the $(i+1)$-st coefficient of the power series $l^{-1}$. Denote by $e \in \Fp[\bb][[x]]$ the image of the power series $\exp \in \Zz[\bb][[x]]$ defined in \eqref{p:def_exp}. By \rref{p:FGL_exp}, the morphism $\pi$ classifies the formal group law $(\Fp[\bb],F)$, where
\[
F(x,y) = e (e^{-1}(x) + e^{-1}(y))
\]
so that the morphism $\varphi \circ \pi$ classifies the formal group law $(\Gamma,\varphi_*F)$, where
\[
\varphi_*F(x,y) = \varphi_*e (\varphi_*(e^{-1})(x) + \varphi_*(e^{-1})(y)).
\]
Here the notation $\varphi_*$ stands for the ring morphism $\Fp[\bb][[x,y]] \to \Gamma[[x,y]]$, resp.\ $\Fp[\bb][[x]] \to \Gamma[[x]]$, induced by taking the image of the coefficients under $\varphi$. By construction $\varphi_*e=l^{-1}$, and $\varphi_*(e^{-1}) =(\varphi_*e)^{-1} = l$, hence
\[
\varphi_*F(x,y) = l^{-1}(l(x) +l(y)) = \Phi,
\]
which proves that $\varphi \circ \pi = \id_\Gamma$. 
\end{proof}

\begin{remark}
The rings $\Laz$ and $\Lp$ admit the following concrete descriptions. Declare two smooth projective $k$-schemes equivalent if they have the same collection of Chern numbers (resp.\ modulo $p$ Chern numbers), indexed by partitions (see \rref{p:Chern_number}). The set of equivalence classes is an abelian monoid for the disjoint union of $k$-schemes. The associated abelian group, together with its ring structure induced by the cartesian product of $k$-schemes coincides with $\Laz$ (resp.\ $\Lp$). In view of \rref{th:Lazard} and \rref{prop:Lp_univ} and \cite[Theorems 6.5 and 7.8]{Quillen-Elementary}, the ring $\Laz$ (resp.\ $\Ld$) is isomorphic to the complex (resp.\ unoriented) cobordism ring.
\end{remark}

\subsection{Projective bundles}
\hfill\\

\vspace*{-1em}
In this section the theory $\Hh$ is either $\CH$ or $\Ch = \CH/p$ for some prime $p$. We will compute the pushforward morphism along a projective bundle in the theory $\Ht$ in terms of Chern classes in the theory $\Hh$. This is a variant of Quillen's formula for complex cobordism \cite[Theorem~1]{Quillen-FGL}.

\begin{para}
Let $X\in \Sm_k$. Denote by $\Rh(X)$ the set of those elements of $\Ht(X)[[y]]$ whose $y^i$-coefficient lies in $\Ht^{-i}(X)$ for all $i\in \Nn$. Then $\Rh(X)$ is a subring of $\Ht(X)[[y]]$. Moreover, if $f \in \Rh(X)$ is invertible in $\Ht(X)[[y]]$, then $f^{-1} \in R(X)$. Using the fact that $\Hh^j(X)=0$ for $j<0$, we see that, for any partition $\alpha$, the $b_\alpha$-coefficient of any element of $\Rh(X)$ is of the form $a_0 + a_1y + \cdots +a_{|\alpha|} y^{|\alpha|}$ with $a_i \in \Hh^{|\alpha| -i}(X)$.
\end{para}

\begin{para}
\label{p:fund_pol_y}
It follows from the splitting principle \rref{p:splitting} that there is a unique way to define for every $X\in \Sm_k$ and $E \in K_0(X)$ an element $P^{\Hh}(E\{y\}) \in \Rh(X)^\times$ satisfying (see \rref{p:fund_pol} for the definition of $\pi$)
\begin{enumerate}[(i)]
\item $f^*P^{\Hh}(E\{y\}) = P^{\Hh}(f^*E\{y\})$ for any  $f\colon Y \to X$ in $\Sm_k$ and $E\in K_0(X)$,

\item $P^{\Hh}(L\{y\}) = \pi(c_1^{\Hh}(L)+y)$ when $L$ is a line bundle,
\item $P^{\Hh}( (E+F)\{y\}) = P^{\Hh}(E\{y\}) P^{\Hh}(F\{y\})$ for any $E,F \in K_0(X)$.
\end{enumerate}
When $E,F \in K_0(X)$ we set $P^{\Hh}(E\{y\} +F)=P^{\Hh}(E\{y\})P^{\Hh}(F) \in \Rh(X)^{\times}$.
\end{para}

\begin{para}
\label{p:fund_pol_tensor}
Let $L$ be a line bundle over $X\in \Sm_k$, and $E,F \in K_0(X)$. It follows from the splitting principle \rref{p:splitting} (and \rref{p:additive}) that $P^{\Hh}(E\otimes L+F) \in \Ht(X)$ is the image of $P^{\Hh}(E\{y\}+F) \in \Ht(X)[[y]]$ under $y \mapsto c_1^{\Hh}(L)$.
\end{para}

\begin{para}
\label{p:E1}
Let $X \in \Sm_k$ and $E,F \in K_0(X)$. For each partition $\alpha$, we denote by $c_\alpha(E\{y\}+F) \in \Hh(X)[y]$ the $b_\alpha$-coefficient of $P^{\Hh}(E\{y\}+F) \in \Rh(X)$. Its image under $y\mapsto 1$ is an element $c_\alpha(E\{1\}+F) \in \Hh(X)$ whose component in $\Hh^j(X)$ we denote by $c_\alpha(E\{1\}+F)_j$. Then in $\Hh(X)[y]$
\begin{equation}
\label{eq:E1}
c_\alpha(E\{y\}+F) = \sum_{i=0}^{|\alpha|} c_\alpha(E\{1\}+F)_{|\alpha|-i} y^i.
\end{equation}
\end{para}

\begin{para}
When $\alpha$ is the partition $(1,\cdots,1)$ with $|\alpha|=n$, we will write $c_n(E\{y\}+F)$ instead of $c_\alpha(E\{y\}+F)$ (see \rref{p:c_alpha_c_i}). 
\end{para}

\begin{lemma}
\label{lemm:CF2}
Let $X\in \Sm_k$ and $E,F\in K_0(X)$. For any partition $\alpha$, we have in $\Hh(X)[y]$ (see \rref{p:CF-Q} and \rref{lemm:c_alpha_inv} for the definitions of $Q_\alpha$ and $\lambda_{\alpha,\beta}$)
\[
Q_\alpha(c_1(E\{y\}+F),c_2(E\{y\}+F),\cdots) = c_\alpha(E\{y\}+F) = \sum_{|\beta|=|\alpha|} \lambda_{\alpha,\beta} c_\beta(-E\{y\}-F).
\]
\end{lemma}
\begin{proof}
Let $n=|\alpha|$. By \dref{def:oct}{def:oct:PB}, the morphism $\rho \colon \Hh(X)[y] \to \Hh(X \times \Pp^n)$ induced by $y \mapsto c_1(\Oc(1))$ and the pullback along $X \times \Pp^n \to X$ restricts to an injection on the subset of polynomials in $y$ of degree $\leq n$. The equalities take place in that subset by \eqref{eq:E1}, hence it suffices to verify their images under $\rho$. By \rref{p:fund_pol_tensor}, this follows from \rref{lemm:CF-Q} and \rref{lemm:c_alpha_inv} applied to $E(1)+F \in K_0(X \times \Pp^n)$.
\end{proof}

\begin{para}
Let $R$ be a commutative ring. We define a morphism $\Res_y \colon R[[y]] \to R$ by mapping a power series $\sum_{i \in \Zz} a_i y^i$ to its $y^{-1}$-coefficient $a_{-1}$.
\end{para}

\begin{proposition}
\label{prop:Quillen}
Let $S \in \Sm_k$ and $V \to S$ be a vector bundle of rank $r$. Denote by $p \colon \Pp(V) \to S$ the associated projective bundle. Then for any $m\geq 0$
\[
p_*^{\Ht} \big( c_1^{\Ht}(\Oc_{\Pp(V)}(1))^m\big) = \sum_{i \in \Nn} c_i^{\Hh}(-V) \Res_y\big( y^{-r-i}(\exp y)^mP^{\Hh}(-V\{y\}) \big)\in \Ht(S).
\]
\end{proposition}
\begin{proof}
Let $\xi = c_1^{\Hh}(\Oc_{\Pp(V)}(1))$. Then $c_1^{\Ht}(\Oc_{\Pp(V)}(1))=\exp \xi$ by \rref{p:c1_tilde}. Write
\[
(\exp y)^mP^{\Hh}(-V\{y\})) = \sum_{j \in \Nn} \varphi_j y^j,
\]
where $\varphi_j \in \Ht(S)$. Since the tangent bundle of $p$ satisfies $\Tan_p = (p^*V)(1) -1 \in K_0(\Pp(V))$ (see e.g.\ \cite[\S B.5.8]{Ful-In-98}), we have in view of \rref{p:fund_pol_tensor}
\[
p_*^{\Ht} \big( c_1^{\Ht}(\Oc_{\Pp(V)}(1))^m\big) = p^{\Hh}_*\big((\exp \xi)^mP^{\Hh}(-(p^*V)(1))\big) = p^{\Hh}_* \Big( \sum_{j \in \Nn}p_{\Hh}^*(\varphi_j) \xi^j \Big).
\]
We have $p^{\Hh}_*(\xi^j) = 0$ for $j < r-1$ (see \cite[Proposition~3.1(a)(i)]{Ful-In-98}) and $p^{\Hh}_*(\xi^j) = c_{j+1-r}^{\Hh}(-V)$ for $j \geq r-1$ (this is how Chern classes are defined in \cite[\S3.2]{Ful-In-98}; that this definition coincides with the one given in \rref{p:Chern} follows from \cite[Remarks 3.2.4 and 3.2.3(a), Propositions~2.5(e) and 2.6(b)]{Ful-In-98}). Using the projection formula  \dref{def:oct}{def:oct:PF}, we obtain the required equality
\[
p_*^{\Ht} \big( c_1^{\Ht}(\Oc_{\Pp(V)}(1))^m\big) = \sum_{j\in\Nn} \varphi_j p^{\Hh}_*(\xi^j) = \sum_{i\in\Nn} \varphi_{i+r-1} c_i^{\Hh}(-V).\qedhere
\]
\end{proof}
\begin{corollary}
\label{cor:Quillen}
Let $S$ be a smooth projective $k$-scheme and $V \to S$ be a vector bundle of rank $r$. Then for any $m\geq 0$, we have in $\Ht(\Spec k)$
\[
\lc c_1^{\Ht}(\Oc_{\Pp(V)}(1))^m \rc = \sum_{i \in \Nn}\Res_y\Big( y^{-r-i} (\exp y)^m\deg \big(c_i^{\Hh}(-V)P^{\Hh}(-V\{y\} - \Tan_S)\big) \Big).
\]
\end{corollary}
\begin{proof}
This follows from \rref{prop:Quillen} by pushing forward along $S \to \Spec k$.
\end{proof}

\begin{corollary}
\label{lemm:projbundle_trivial}
Let $V \to S$ be a vector bundle of rank $r$, and $p\colon \Pp(V) \to S$ the associated projective bundle. If $c^{\Hh}_i(V)=0 \in \Hh(S)$ for all $i>0$, then for any $m \geq 0$ (we write $\Pp^n=\varnothing$ when $n<0$)
\[
p_*^{\Ht} \big( c_1^{\Ht}(\Oc_{\Pp(V)}(1))^m\big) = \lc \Pp^{r-1-m} \rc \in \Ht(S).
\]
\end{corollary}
\begin{proof}
By its construction \rref{p:fund_pol_y}, the element $P^{\Hh}(V\{y\})$ depends only on $r$ and the Chern classes $c_i^{\Hh}(V)$. Thus, it follows from \rref{prop:Quillen} that $p_*^{\Ht}(c_1^{\Ht}(\Oc_{\Pp(V)}(1))^m)$ depends only on $r,m$ and the Chern classes $c_i^{\Hh}(V)$. Therefore we may assume that the bundle $V$ is trivial, and the statement is clear (see \cite[Lemma~5.2]{Mer-Ori}).
\end{proof}

\section{\texorpdfstring{$\mud$}{\textmu 2}-actions}
\numberwithin{theorem}{section}
\numberwithin{lemma}{section}
\numberwithin{proposition}{section}
\numberwithin{corollary}{section}
\numberwithin{example}{section}
\numberwithin{notation}{section}
\numberwithin{definition}{section}
\numberwithin{remark}{section}

\label{sect:mud}
\begin{para}
The functor associating to each commutative $k$-algebra $R$ the subgroup of those $r \in R^\times$ such that $r^2=1$ is represented by a finite commutative algebraic group $\mud$. We refer e.g.\ to \cite[I]{SGA3-1} for the notion of $\mud$-action on a quasi-projective $k$-scheme $X$. In the affine case $X=\Spec A$, this is the same thing as a $\Zz/2$-grading $A=A_0 \oplus A_1$ as $k$-algebra \cite[I, 4.7.3.1]{SGA3-1}. In general, the scheme $X$ is covered by affine $\mud$-invariant open subschemes \cite[V, \S5]{SGA3-1}.
\end{para}

\begin{para}
Let $X$ be a quasi-projective $k$-scheme with a $\mud$-action. An open or closed subscheme $Y$ of $X$ is called \emph{$\mud$-invariant} if its inverse images under the projection and the action $\mud \times X \to X$ coincide.
\end{para}

\begin{para}
\label{p:quotient}
Let $X$ be a quasi-projective $k$-scheme with a $\mud$-action. The equaliser of the projection and the action $\mud \times X \to X$ is represented by a finite surjective morphism $\varphi \colon X \to X/\mud$, called the \emph{quotient morphism} (see \cite[V, Th\'eor\`eme 4.1]{SGA3-1}). The $k$-scheme $X/\mud$ is quasi-projective by \cite[V, Remarque 5.1]{SGA3-1}. The $\Oc_{X/\mud}$-algebra $\Ac = \varphi_*\Oc_X$ admits a $\Zz/2$-grading $\Ac=\Ac_0 \oplus \Ac_1$, where $\Oc_{X/\mud} = \Ac_0$ (see e.g.\ \cite[(3.2.2)]{isol}).
\end{para}

\begin{para}
Let $X$ be a quasi-projective $k$-scheme with a $\mud$-action. The functor associating to a quasi-projective $k$-scheme $T$ with trivial $\mud$-action the set of $\mud$-equivariant morphisms $T \to X$ is represented by a $\mud$-invariant closed subscheme $X^\mud$ of $X$, called the \emph{fixed locus}. Its ideal $\Ic \subset \Oc_X$ is characterised by the fact that the ideal $\varphi_*\Ic \subset \Oc_{X/\mud}$ is generated by $\Ac_1$ (using the notation of \rref{p:quotient}).
\end{para}

It is possible to provide a more concrete definition of the notion of $\mud$-action, by distinguishing cases according to the characteristic of the base field:

\begin{proposition}
\label{prop:mud_char}
Let $X$ be a quasi-projective $k$-scheme.
\begin{enumerate}[(i)]
\item \label{prop:mud_char:1} Assume that $\carac k \neq 2$. Then a $\mud$-action on $X$ is the same thing as a $k$-morphism $\sigma \colon X \to X$ such that $\sigma^2=\id_X$. The fixed locus $X^\mud$ is the equaliser of the morphisms $\id_X$ and $\sigma$.

\item \label{prop:mud_char:2} Assume that $\carac k = 2$. Then a $\mud$-action on $X$ is the same thing as a $k$-derivation $D \colon \Oc_X \to \Oc_X$ satisfying $D \circ D=D$. The fixed locus $X^\mud$ is the vanishing locus of the section $D' \in H^0(X,\Omega_{X/k}^\vee)$ corresponding to $D$.
\end{enumerate}
\end{proposition}
\begin{proof}
A $\mud$-action on $X$ is given by a $k$-morphism $\mud \times X \to X$ satisfying certain conditions. In case \eqref{prop:mud_char:1} the morphism $\mud \times X = X \sqcup X \to X$ is given by $\id_X \sqcup \sigma$, while in case \eqref{prop:mud_char:2} the morphism $\mud \times X = \Spec (k[\varepsilon]/\varepsilon^2) \times X \to X$ is given by the pair consisting of $\id_X$ and $D' \in H^0(X,\Omega_{X/k}^\vee)$. To verify the remaining statements, we may assume that $X = \Spec A$.

\eqref{prop:mud_char:1}: The correspondence between the grading $A=A_0 \oplus A_1$ and the involution $s \colon A \to A$ is given by the following formulas:
\[
A_0 = \ker(s - \id) = \im(s+\id), \quad A_1 = \ker(s +\id) = \im(s-\id), \quad s(a) = a_0 -a_1,
\]
where $a_0 \in A_0, a_1 \in A_1$ are the components of an arbitrary element $a\in A$. The coequaliser of the ring morphisms $\id_A$ and $s$ is the quotient of $A$ by the ideal generated $\im(s -\id)=A_1$, whence the given description of $X^\mud$.

\eqref{prop:mud_char:2}: The correspondence between the grading $A=A_0 \oplus A_1$ and the derivation $\partial \colon A \to A$ is given by the following formulas:
\[
A_0 = \ker \partial, \quad A_1 = \im \partial, \quad \partial(a) = a_1,
\]
where $a_1 \in A_1$ is the component of an arbitrary element $a\in A$. The section $D' \in H^0(X,\Omega_{X/k}^\vee)$ is given by the unique $A$-module morphism $\partial' \colon \Omega_{A/k} \to A$ satisfying $\partial = \partial' \circ d$, where $d \colon A \to \Omega_{A/k}$ is the universal derivation. The vanishing locus of $D'$ is the closed subscheme defined by the ideal $J$ generated by $\im\partial'$. Since the $A$-module $\Omega_{A/k}$ is generated by $\im d$, it follows that $J$ is generated by $\im(\partial' \circ d) = \im\partial = A_1$, whence the given description of $X^\mud$.
\end{proof}

We will repeatedly use the next lemma without explicit mention.
\begin{lemma}
\label{lemm:fix_sm}
Let $X$ be a smooth quasi-projective $k$-scheme with a $\mud$-action. Then the fixed locus $X^\mud$ is smooth over $k$.
\end{lemma}
\begin{proof}
See e.g.\ \cite[Lemma~3.5.2]{isol}.
\end{proof}

\begin{lemma}
\label{lemm:inv_blowup}
Let $X$ be a quasi-projective $k$-scheme with a $\mud$-action. The blowup $Y$ of $X^\mud$ in $X$ inherits a $\mud$-action whose fixed locus $Y^\mud$ is the exceptional divisor.
\end{lemma}
\begin{proof}
Let $E$ be the exceptional divisor in $Y$. Denote by $a \colon \mud \times X \to X$ the action morphism. Since the closed subscheme $X^\mud$ of $X$ is $\mud$-invariant, its inverse image under the composite $\mud \times Y \to \mud \times X \xrightarrow{a} X$ is the closed subscheme $\mud \times E$. The existence of the morphism $\mud \times Y \to Y$ and the fact that it is a group action then follow from the universal property of the blowup.

To check that $Y^\mud=E$, we may assume that $X=\Spec A$. Let $I$ be the ideal of $A$ generated by $A_1$. For $s\in A_1$, consider the $A$-algebra $S_{(s)} = \{as^{-n}|a\in I^n, n \in \Nn\} \subset A[s^{-1}]$, with its induced $\Zz/2$-grading. The scheme $Y$ is covered by the open subschemes $D(s) = \Spec (S_{(s)})$ for $s\in A_1$, and it follows from the universal property of the blowup that the immersions $D(s) \to Y$ are $\mud$-equivariant. Let $s\in A_1$. Any element $x \in S_{(s)}$ homogeneous of degree $1 \in \Zz/2$ may be written as $as^{-n}$ where $n\in \Nn$ and $a \in I^n \cap A_{1+r(n)}$ (we denote by $r\colon \Zz \to \Zz/2$ the reduction modulo $2$). But 
\[
I^n \cap A_{1+r(n)} = ((A_1)^nA) \cap A_{1+r(n)}= (A_1)^nA_1 \subset I^{n+1},
\]
so that $as^{-n-1} \in S_{(s)}$, hence $x \in s S_{(s)}$. Thus $D(s)^{\mud} = E \cap D(s)$, and $Y^\mud = E$.
\end{proof}

\begin{lemma}
\label{lemm:inv_cod1}
Let $Y$ be a quasi-projective $k$-scheme with a $\mud$-action such that $Y^\mud \to Y$ is an effective Cartier divisor. Denote by $f\colon Y \to Z=Y/\mud$ the quotient morphism.
\begin{enumerate}[(i)]
\item \label{lemm:inv_cod1:1} The $\Oc_Z$-module $\Lc = \coker (\Oc_Z \to f_*\Oc_Y)$ is locally free of rank one.

\item \label{lemm:inv_cod1:2} There is a canonical isomorphism $f^*\Lc^{\vee} \simeq \Oc_Y(Y^\mud)$.

\item \label{lemm:inv_cod1:3} The morphism $Y^\mud \to Z$ is an effective Cartier divisor whose ideal is isomorphic to $\Lc^{\otimes 2}$.

\item \label{lemm:inv_cod1:4} If $Y$ is smooth over $k$, then so is $Z$.
\end{enumerate}
\end{lemma}
\begin{proof}
As recalled in \rref{p:quotient}, there is a $\Zz/2$-grading of the $\Oc_Z$-algebra $f_*\Oc_Y=\Ac = \Ac_0 \oplus \Ac_1$ such that $\Oc_Z=\Ac_0$. Thus $\Lc = \Ac_1$. We let $\Ic = \Oc_Y(-Y^\mud)$.

\eqref{lemm:inv_cod1:1} : We may replace $Y$ with any cover by $\mud$-invariant open subschemes. In particular, we may assume that $Y=\Spec A$ and that the closed subscheme $Y^\mud$ of $Y$ is defined by the ideal $I=A_1A = aA$ of $A$, for some nonzerodivisor $a\in A$. Denote by $a_0 \in A_0$ and $a_1 \in A_1$ the components of $a$, and write $a_1=ua$ with $u \in A$. Since $a_0 =a -a_1 \in I \cap A_0 \subset I^2$, we may find $w \in A$ such that $a_0 = wa^2$. Then $a=wa^2 + ua$, and since $a$ is a nonzerodivisor in $A$, we have $1=wa +u \in A_1A+uA$. Thus $Y$ is covered by the open subschemes $D(f) = \Spec (A[f^{-1}])$ for $f\in A_1$ and $D(u) = \Spec (A[u^{-1}])$. The subschemes $D(f)$ are $\mud$-invariant. So is $D(u)$, because it is the locus in $Y$ where the $\mud$-invariant closed subschemes $Y^\mud$ and $\Spec (A/a_1A)$ coincide (alternatively, one may check directly that the ideal $uA$ of $A$ is homogeneous). Therefore we may assume either that $A_1$ contains an element $f \in A^\times$, or that $u\in A^\times$.

If $f \in A_1 \cap A^\times$, then $f^{-1} \in A_1$, and $x \mapsto fx$ induces an isomorphism of $A_0$-modules $A_0 \to A_1$, proving that $A_1$ is free or rank one.

Assume that $u\in A^\times$. Since $a$ is a nonzerodivisor in $A$, so is $a_1=ua$. Thus the morphism of $A_0$-modules $A_0 \to A$ given by $x \mapsto a_1 x$ is injective; its image $a_1 A_0 = (a_1 A)\cap A_1 = I\cap A_1 = A_1$ is a free $A_0$-module of rank one.

\eqref{lemm:inv_cod1:2} : The morphism of $\Oc_Y$-modules $\alpha \colon f^*\Lc=f^*\Ac_1 \to \Ic$ adjoint to the inclusion $\Ac_1 \subset f_*\Ic$ is surjective because $f_*\alpha \colon f_*f^*\Ac_1 \to f_*\Ic$ is surjective, since the ideal $f_*\Ic$ of $\Ac$ is generated by $\Ac_1$. The morphism $\alpha$ must be an isomorphism, because its source and target are locally free modules of rank one by \eqref{lemm:inv_cod1:1}. 

\eqref{lemm:inv_cod1:3} : The affine morphism $Y^\mud \to Z$ is given by the morphism of $\Oc_Z$-algebras $\Oc_Z=\Ac_0 \to \Ac \to \Ac/f_*\Ic$. This morphism is surjective with kernel $f_*\Ic \cap \Ac_0 = (\Ac_1)^2$, the image of the morphism of $\Oc_Z$-modules $\beta\colon \Lc^{\otimes 2} = (\Ac_1)^{\otimes 2} \to \Ac_0$ induced by the $\Zz/2$-graded $\Ac_0$-algebra structure on $\Ac$. Since the $\Oc_Z$-module $\Lc^{\otimes 2}$ is locally free of rank one by \eqref{lemm:inv_cod1:1}, in order to prove \eqref{lemm:inv_cod1:3}, it will suffice to prove that the $\Oc_Z$-module $(\Ac_1)^2$ is locally free of rank one (then $\beta$ will have to be an isomorphism). To do so, we may assume that $Y=\Spec A$ and moreover, in view of \eqref{lemm:inv_cod1:1}, that $A_1=lA_0$, for some $l\in A_1$. Then the ideal $A_1A$ of $A$ is invertible by assumption, hence its generator $l$ must be a nonzerodivisor in $A$. Then $l^2 \in A_0$ is a nonzerodivisor in $A$, hence in its subring $A_0$. Thus $(A_1)^2=l^2A_0$ is an invertible ideal of $A_0$.

\eqref{lemm:inv_cod1:4} : By \eqref{lemm:inv_cod1:1}, the morphism $f \colon Y \to Z$ is faithfully flat, so that the statement follows from \cite[(17.7.7)]{ega-4-4}.
\end{proof}

\section{Cobordism and fixed locus}
\numberwithin{theorem}{subsection}
\numberwithin{lemma}{subsection}
\numberwithin{proposition}{subsection}
\numberwithin{corollary}{subsection}
\numberwithin{example}{subsection}
\numberwithin{notation}{subsection}
\numberwithin{definition}{subsection}
\numberwithin{remark}{subsection}

\label{sect:parity}
\subsection{Parity of Chern numbers}
\begin{lemma}
\label{lemm:pure}
Let $Y$ be a smooth projective $k$-scheme with a $\mud$-action such that $Y^\mud$ has pure codimension one in $Y$. Then for any $m \geq 0$
\[
\lc c_1(\Oc_Y(Y^\mud))^m \rc = 0 \in \Ld.
\]
\end{lemma}
\begin{proof}
Let $f\colon Y \to Z=Y/\mud$ be the quotient morphism. Then $Z \in \Sm_k$ by \dref{lemm:inv_cod1}{lemm:inv_cod1:4} (and \rref{lemm:fix_sm}). As $[2]_{\Cht}(x) =0 \in \Ld[[x]]$ by \rref{lemm:FGL_Lp}, applying \rref{lemm:Vishik} yields $f_*[Y] = 0 \in \Cht(Z)$. Since $\Oc_Y(Y^\mud)$ is the pullback of a line bundle on $Z$ by \dref{lemm:inv_cod1}{lemm:inv_cod1:1} and \dref{lemm:inv_cod1}{lemm:inv_cod1:2}, the projection formula \dref{def:oct}{def:oct:PF} implies that $f_*(c_1(\Oc_Y(Y^\mud))^m) = 0 \in \Cht(Z)$, and the lemma follows by pushing forward along $Z \to \Spec k$.
\end{proof}

\begin{remark}
Let $X$ be a connected smooth projective $k$-scheme with a nontrivial $\mud$-action. Lemma \rref{lemm:pure} (for $m=0,1$) implies that if one Chern number of $X$ or of $X^\mud$ is odd, then $X^\mud$ must have a component of dimension $\leq \dim X -2$.
\end{remark}

\begin{theorem}
\label{th:X_PN_L2}
Let $X$ be a smooth projective $k$-scheme with a $\mud$-action, and $N$ the normal bundle to the immersion of the fixed locus $X^\mud \to X$. Then in $\Laz_2$
\[
\lc \Pp(N \oplus 1) \rc = \lc X \rc \quad \text{ and } \quad \lc c_1(\Oc_{\Pp(N\oplus 1)}(1))^m \rc=0 \quad \text{ for $m\geq 1$ }.
\]
\end{theorem}
\begin{proof}
Let $Y$ be the blowup of $X^\mud$ in $X$. Then $Y^\mud = \Pp(N)$ has pure codimension one in $Y$ by \rref{lemm:inv_blowup}. It follows from \rref{lemm:FGL_Lp} that $[-1]_{\Cht}(x)=x \in \Ld[[x]]$, hence $c_1(\Oc_{\Pp(N\oplus 1)}(-1)) = c_1(\Oc_{\Pp(N\oplus 1)}(1))$ in $\Cht(\Pp(N\oplus 1))$. The statement now follows from \rref{lemm:deformation} (with $g=x^m$ for $m\geq 0$) and \rref{lemm:pure}.
\end{proof}

As a sample application of \rref{th:X_PN_L2}, we deduce an algebraic version of a theorem of Conner--Floyd \cite[(25.1)]{CF-book-1st}.

\begin{corollary}
\label{cor:trivial_normal}
Let $X$ be a smooth projective $k$-scheme with a $\mud$-action, and $N$ the normal bundle to the immersion of the fixed locus $X^\mud \to X$. Assume that $X^\mud$ contains no component of $X$, and that $c_i^{\CH}(N) \in 2\CH(X^\mud)$ for all $i>0$. Then every Chern number of $X$ or of $X^\mud$ is even.
\end{corollary}
\begin{proof}
Write $X^\mud$ as the disjoint union of subschemes $F^r$ having pure codimension $r$ in $X$, for $r=0,\cdots,\dim X$. By \rref{lemm:projbundle_trivial} (for $\Hh=\CH/2$) and \rref{th:X_PN_L2}, for any $m\geq 0$ we have in $\Ld$
\[
\lc c_1(\Oc_{\Pp(N\oplus 1)}(1))^m \rc = \sum_{r=m}^{\dim X} \lc \Pp^{r-m} \rc \lc F^r \rc = 
\begin{cases}
\lc X \rc &\text{ if $m=0$,} \\
0 &\text{ if $m>0$.}
\end{cases}
\]
By descending induction on $r$ we deduce that $\lc F^r\rc=0$ for $r \geq 1$, and that $\lc X \rc = \lc F^0 \rc$. Since by assumption $F^0 =\varnothing$, it follows that $\lc X\rc=0$, and that $\lc X^\mud \rc=\lc F^1 \rc +\cdots + \lc F^{\dim X}\rc =0$ in $\Ld$.
\end{proof}

\subsection{Kosniowski--Stong formula}
\hfill\\

\vspace*{-1em}

In this section, we consider the theory $\Hh=\CH/2$ and use the notation of \rref{p:Chern_virtual} and \rref{p:E1}.

\begin{proposition}
\label{prop:KS}
Let $X$ be a smooth projective $k$-scheme of pure dimension $n$ with a $\mud$-action, and $N$ the normal bundle to the immersion of the fixed locus $X^\mud \to X$. Let $\alpha$ be a partition such that $|\alpha| \leq n$. Then
\[
c_\alpha(X) = \deg \big(c(-N)c_\alpha(-N\{1\} - \Tan_{X^\mud})\big) \in \Fd.
\]
\end{proposition}
\begin{proof}
Write $X^\mud$ as the disjoint union of the schemes $F^0,\cdots,F^n$, where $F^r$ has pure dimension $n-r$, and let $N^r=N|_{F^r}$. Consider the Laurent series
\[
f(y) = \sum_{r=0}^n \sum_{i=0}^{n-r} y^{-r-i} \deg \big(c_i(-N^r)P^{\Hh}(-N^r\{y\} -\Tan_{F^r})\big) \in \Fd[\bb][[y]][y^{-1}].
\]
Since $P^{\Hh}(-1\{y\}) = \pi(y)^{-1} = (\exp y)^{-1}y$ and $c_i(-N^r) = c_i(-(N^r\oplus 1))$, we have
\[
(\exp y)^{-1} f(y) = \sum_{r=0}^n \sum_{i=0}^{n-r} y^{-r-1-i} \deg \big(c_i(-(N^r\oplus 1)))P^{\Hh}(-(N^r\oplus 1)\{y\} -\Tan_{F^r})\big).
\]
Therefore by \rref{cor:Quillen}, we have for any $m\geq 0$
\[
\Res_y ( (\exp y)^{m-1} f(y))= \lc c_1(\Oc_{\Pp(N\oplus 1)}(1))^m \rc \in \Fd[\bb].
\]
By \rref{th:X_PN_L2}, this element vanishes when $m>0$, and equals $\lc X \rc$ when $m=0$. Now
\[
(\exp y)^{m-1}  = y^{m-1}\pi(y)^{m-1}= y^{m-1} + \sum_{s\geq m} g_{s,m}y^s  \in \Fd[\bb][[y]][y^{-1}]
\]
where $g_{s,m} \in \Fd[\bb]$. By descending induction on $m$ (the case $m>n$ being clear from the definition of $f(y)$), we obtain in $\Fd[\bb]$
\[
\Res_y (y^{m-1}f(y)) =  
\begin{cases}
\lc X \rc &\text{ if $m=0$,} \\
0 &\text{ if $m>0$.}
\end{cases}
\]
We consider the $b_\alpha$-coefficient of this equation for $m=n-|\alpha|\geq 0$. Since $c_\alpha(X) =0$ if $|\alpha| \neq n$, in view of \rref{p:Chern_Laz} and \eqref{eq:E1} we obtain in $\Fd$
\begin{align*}
c_\alpha(X)
&=\sum_{r=0}^n \sum_{i=0}^{n-r} \sum_{j=0}^{|\alpha|} \Res_y \Big(y^{j+n-|\alpha|-1-r-i} \deg \big(c_i(-N^r)c_\alpha(-N^r\{1\} -\Tan_{F^r})_{|\alpha|-j}\big)\Big)\\
&= \sum_{r=0}^n \sum_{i=0}^{n-r} \deg \big(c_i(-N^r)c_\alpha(-N^r\{1\} -\Tan_{F^r})_{n-r-i}\big)\\
&= \deg \big(c(-N)c_\alpha(-N\{1\} -\Tan_{X^\mud})\big).\qedhere
\end{align*}
\end{proof}

We obtain the following analogue of the Kosniowski--Stong formula \cite{KS}:
\begin{corollary}
\label{cor:KS}
Let $X$ be a smooth projective $k$-scheme of pure dimension $n$ with a $\mud$-action, and $N$ the normal bundle to $X^\mud \to X$. Let $f\in \Zz[y_1,\cdots,y_n]$ be a polynomial of total degree $\leq n$, where each $y_i$ has degree $i$. Then in $\Fd$
\[
\deg \big(f(c_1(\Tan_X),\cdots,c_n(\Tan_X))\big) = \deg \big(c(-N)f(c_1(N\{1\}+\Tan_{X^{\mud}}),\cdots,c_n(N\{1\}+\Tan_{X^{\mud}}))\big).
\]
\end{corollary}
\begin{proof}
By \rref{p:Q-basis} we may assume that $f=Q_\alpha$ with $|\alpha| \leq n$. In view of \rref{lemm:CF2}, the statement follows from \rref{prop:KS}.
\end{proof}

\subsection{Cobordism modulo two}

\begin{theorem}
\label{th:Lmod2}
Let $X$ be a smooth projective $k$-scheme with a $\mud$-action, and $N$ the normal bundle to the immersion of the fixed locus $X^\mud \to X$. Write $\zeta=c_1(\Oc_{\Pp(N\oplus 1)}(-1)) \in \CHt(\Pp(N\oplus 1))$. Then in $\Laz/2$ (using the convention of \rref{p:frac})
\[
\lc X \rc = \Big \lc \frac{2\zeta}{[2]_{\CHt}(\zeta)} \Big\rc \quad \text{ and } \quad 0 = \Big\lc \frac{2\zeta^{m+1}}{[2]_{\CHt}(\zeta)} \Big\rc \quad \text{ for $m\geq 1$}.
\]
\end{theorem}
Before proving the theorem, let us clarify its statement. There is a unique power series $v(x)\in \Laz[2^{-1}][[x]]$ such that $v(x)\cdot [2]_{\CHt}(x) =x$. The theorem says that $\lc 2v(\zeta) \rc \in \Laz[2^{-1}]$ belongs to $\Laz$ and is congruent to $\lc X\rc$ modulo $2\Laz$, and that $\lc \zeta^m v(\zeta) \rc \in \Laz[2^{-1}]$ belongs to $\Laz$ for $m\geq 1$. 

\begin{proof}
Let $Y$ be the blowup of $X^\mud$ in $X$, and $f\colon Y \to Z=Y/\mud$ the quotient morphism. Then by \rref{lemm:inv_blowup} we have $Y^\mud = \Pp(N)$, and by \rref{lemm:inv_cod1} the $\Oc_Z$-module $\Lc = \coker(\Oc_Z \to f_*\Oc_Y)$ is invertible and $Z \in \Sm_k$. Let $\mu = c_1(\Lc^\vee) \in \CHt(Z)$ and $\eta = c_1(\Oc_Y(\Pp(N))) \in \CHt(Y)$. Then $\eta = f^* \mu$ by \dref{lemm:inv_cod1}{lemm:inv_cod1:2}. It follows form the projection formula \dref{def:oct}{def:oct:PF} and \rref{lemm:Vishik} (pushing forward along $Z \to \Spec k$) that
\begin{equation}
\label{eq:eta_2}
\Big \lc \frac{\eta^{m+1}}{[2]_{\CHt}(\eta)} \Big\rc = \lc \mu^m \rc \in \Laz[2^{-1}] \quad \text{ for any $m\geq 0$}.
\end{equation}
Applying \rref{lemm:deformation} with $g=x^{m+1}([2]_{\CHt}(x))^{-1}$ for $m\geq 1$ (and $\Hh=\CHt(-)[2^{-1}]$) yields (the elements $a_i \in \Laz$ are defined in \eqref{p:formal_mult})
\[
\Big \lc \frac{\eta^{m+1}}{[2]_{\CHt}(\eta)} \Big\rc = \Big \lc \frac{\zeta^m[-1]_{\CHt}(\zeta)}{[2]_{\CHt}(\zeta)} \Big\rc =-\Big\lc \frac{\zeta^{m+1}}{[2]_{\CHt}(\zeta)} \Big\rc + \sum_{i \geq 1} a_i\Big\lc \frac{\zeta^{m+1+i}}{[2]_{\CHt}(\zeta)} \Big\rc  \in \Laz[2^{-1}].
\]
This element belongs to $\Laz$ by \eqref{eq:eta_2}, and the second statement follows by descending induction on $m$ (the case $m>\dim X$ being clear). We now apply \rref{lemm:deformation} with $g=2x([2]_{\CHt}(x))^{-1}$ and obtain in $\Laz[2^{-1}]$
\[
2\Big \lc \frac{\eta}{[2]_{\CHt}(\eta)} \Big\rc = \lc X \rc + 2\Big \lc \frac{[-1]_{\CHt}(\zeta)}{[2]_{\CHt}(\zeta)} \Big\rc = \lc X \rc - \Big \lc \frac{2\zeta}{[2]_{\CHt}(\zeta)} \Big\rc +  \sum_{i \geq 1} a_i\Big\lc \frac{2\zeta^{i+1}}{[2]_{\CHt}(\zeta)} \Big\rc.
\]
This element belongs to $2\Laz$ by \eqref{eq:eta_2} (with $m=0$), and so do the terms being summed over $i$ by the second statement. The first statement follows.
\end{proof}

\section{Euler number}
\numberwithin{theorem}{subsection}
\numberwithin{lemma}{subsection}
\numberwithin{proposition}{subsection}
\numberwithin{corollary}{subsection}
\numberwithin{example}{subsection}
\numberwithin{notation}{subsection}
\numberwithin{definition}{subsection}
\numberwithin{remark}{subsection}

\label{sect:Euler}

\subsection{The theory {\texorpdfstring{$\CHx$}{E}}}
\begin{para}
The \emph{Euler number} of a smooth projective $k$-scheme of pure dimension $n$ is
\[
\chi(X) = \deg c_n(\Tan_X) \in \Zz.
\]
As mentioned in the introduction, this integer is the alternate sum of the $\ell$-adic Betti numbers, but we will not use this description.
\end{para}

\begin{definition}
We consider the functor $\CHx = \CHt \otimes_{\Zz[\bb]} \Zz[t]$ where $\Zz[\bb] \to \Zz[t]$ is the morphism $b_i \mapsto (-1)^it^i$, and $t$ has degree $-1$. It follows from \rref{prop:tilde_is_oct}, \rref{ex:tensor_oct} and \rref{p:CH_is_oct} that $\CHx$ is an oriented cohomology theory. We have $\CHx(X)=\CH(X)[t]$ for any $X \in \Sm_k$, and $f_{\CHx}^* = f_{\CH}^*$ for every morphism $f \colon Y \to X$ in $\Sm_k$. If $f$ is projective with virtual tangent bundle $\Tan_f \in K_0(Y)$, then for any $a \in \CHx(Y)$,
\[
f^{\CHx}_*(a)= f^{\CH}_*\Big( a\sum_{i \in \Nn} t^ic_i^{\CH}(\Tan_f)\Big) \in \CHx(X).
\]
\end{definition}

\begin{para}
\label{p:c1_CHx}
Let $L$ be a line bundle over $X \in \Sm_k$. Then $\pi(x) \in \Zz[\bb][[x]]$ maps to $(1+tx)^{-1} \in \Zz[t][[x]]$, hence in view of \rref{p:c1_tilde}
\[
c_1^{\CHx}(L) = \frac{c_1^{\CH}(L)}{1+t c_1^{\CH}(L)}, \quad \text{ hence } \quad c_1^{\CH}(L) = \frac{c_1^{\CHx}(L)}{1-t c_1^{\CHx}(L)}.
\]
\end{para}

\begin{lemma}
\label{lemm:FGL_CHx}
The formal group law of the theory $\CHx$ is given by
\[
x +_{\CHx} y  = \frac{x+y-2txy}{1-t^2xy} \in \Zz[t][[x,y]].
\]
\end{lemma}
\begin{proof}
Let $L,M$ be line bundles over $X \in \Sm_k$. Write $l=c_1^{\CH}(L),m=c_1^{\CH}(M)$ and $\lambda = c_1^{\CHx}(L), \mu = c_1^{\CHx}(M)$ in $\CHx(X)$. By \rref{p:additive} and \rref{p:c1_CHx}, we have in $\CHx(X)$
\begin{align*}
c_1^{\CHx}(L \otimes M) 
&= \frac{l+m}{1+t(l+m)} = \frac{\lambda(1-t\lambda)^{-1} + \mu(1 -t \mu)^{-1}}{1+t(\lambda(1-t\lambda)^{-1} + \mu(1 -t \mu)^{-1})}\\
&= \frac{\lambda(1-t\mu) + \mu(1 -t \lambda)}{(1 -t \lambda)(1-t\mu) +t\lambda(1-t\mu) +t  \mu(1 -t \lambda)}= \frac{\lambda + \mu -2t\lambda\mu}{1-t^2\lambda \mu}.\qedhere
\end{align*}
\end{proof}

\begin{lemma}
\label{lemm:CHx_FGL}
The formal multiplication by $a\in \Zz$ of the theory $\CHx$ is given by
\[
[a]_{\CHx}(x) = \frac{ax}{1+(a-1)tx} \in \Zz[t][[x]].
\]
\end{lemma}
\begin{proof}
Let $L$ be a line bundle over $X \in \Sm_k$. Write $l = c_1^{\CH}(L)$ and $\lambda = c_1^{\CHx}(L)$ in $\CHx(X)$. In view of \rref{p:additive} and \rref{p:c1_CHx}, we have in $\CHx(X)$
\[
c_1^{\CHx}(L^{\otimes a}) =  \frac{al}{1+atl} = \frac{a\lambda(1 -t\lambda)^{-1}}{1+at\lambda(1 -t\lambda)^{-1}}=\frac{a\lambda}{1+(a-1)t\lambda}.\qedhere
\]
\end{proof}

\begin{para}
When $X$ is a smooth projective $k$-scheme of pure dimension $n$, we have
\[
\lc X\rc = \chi(X)t^n \in \fund{\CHx} \subset \CHx(\Spec k)=\Zz[t].
\]
\end{para}

\begin{lemma}
\label{lemm:chi_odd}
The subring $\fund{\CHx} \subset \Zz[t]$ is generated by $2t$ and $t^2$.
\end{lemma}
\begin{proof}
By \rref{cor:Lazard}, this subring is generated by the coefficients $a_{i,j}$ of $F_{\CHx}(x,y)$. Thus the statement follows from \rref{lemm:FGL_CHx}, which implies that
$a_{i,j}$ equals
\[
-2t^{2i-1} \text{ if } i=j>0, \quad t^{2i} \text{ if } i=j-1>0, \quad t^{2j} \text{ if } j=i-1>0, \quad  0 \text{ otherwise}.\qedhere
\]
\end{proof}

\begin{para}
\label{p:2Cc} Lemma \rref{lemm:chi_odd} implies that $2\Zz[t] \subset \fund{\CHx}$ and  $t^2\fund{\CHx} \subset \fund{\CHx}$.
\end{para}

\begin{lemma}
\label{lemm:chi_P}
Let $V$ be a vector bundle of rank $r>0$ over $S\in \Sm_k$ and $p\colon \Pp(V) \to S$ the associated projective bundle. Then
\[
p^{\CHx}_*(1) = rt^{r-1} \in \CHx(S).
\]
\end{lemma}
\begin{proof}
For an element $F \in K_0(S)$, let us denote by $Q(F\{y\}) \in \CHx(S)[y]$ the image of $P^{\CH}(F\{y\}) \in \CHt(S)[y]$ (see \rref{p:fund_pol_y}). We claim that
\[
Q(-V\{y\}) = t^r y^r + (c_1^{\CH}(V) t^r + rt^{r-1}) y^{r-1} + q,
\]
where $q \in \CHx(S)[y]$ is a polynomial in $y$ of degree $\leq r-2$. To see this, by the splitting principle \rref{p:splitting} we may assume that $V$ admits a filtration by subbundles with successive quotients line bundles $L_1,\cdots,L_r$. Then
\[
Q(-V\{y\}) = \prod_{i=1}^r Q(-L_i\{y\}) = \prod_{i=1}^r (1+tc_1^{\CH}(L_i) +ty),
\]
from which the claimed formula follows. Thus \rref{prop:Quillen} (with $m=0$) implies that
\[
p^{\CHx}_*(1) = c_1^{\CH}(V)t^r + rt^{r-1} + c_1^{\CH}(-V)t^r  = rt^{r-1}.\qedhere
\]
\end{proof}

\subsection{Euler number and fixed locus}
\hfill\\

\vspace*{-1em}

The first part of the next statement is well-known, at least when $\carac k \neq 2$.
\begin{proposition}
\label{prop:E_fix}
Let $X$ be a smooth projective $k$-scheme of pure dimension $n$ with a $\mud$-action.
\begin{enumerate}[(i)]
\item We have $\chi(X) = \chi(X^\mud) \mod 2$.

\item \label{prop:E_fix:2} If $n$ is odd, then we have $\chi(X) = \chi(X^\mud) \mod 4$.
\end{enumerate}
\end{proposition}
\begin{proof}
Write $X^\mud$ as the disjoint union of the schemes $F^0,\cdots,F^n$, where $F^r$ has pure dimension $n-r$. Let $N$ be the normal bundle to $X^\mud \to X$ and write $\xi = c_1^{\CHx}(\Oc_{\Pp(N\oplus1)}(1)) \in \CHx(\Pp(N\oplus1))$. By \rref{th:Lmod2} and \rref{lemm:CHx_FGL}, in view of \eqref{p:2Cc} we have in $\fund{\CHx}^{-n}/2$ 
\[
\lc X \rc = \Big\lc 2\frac{[-1]_{\CHx}(\xi)}{[-2]_{\CHx}(\xi)} \Big\rc = \Big\lc2 \frac{-\xi(1-2t\xi)^{-1}}{-2\xi(1-3t\xi)^{-1}} \Big\rc =\Big\lc \frac{1-3t\xi}{1-2t\xi} \Big\rc = \lc \Pp(N\oplus 1) \rc - t\lc \xi \rc.
\]
Now, by \rref{lemm:chi_P}, we have in $\fund{\CHx}^{-n}$
\[
\lc \Pp(N \oplus 1) \rc = \sum_{r=0}^n t^r (r+1) \lc F^r \rc \quad \text{ and } \quad \lc \xi \rc = \lc \Pp(N) \rc = \sum_{r=1}^n t^{r-1} r \lc F^r \rc,
\]
hence $\lc X \rc = \lc F^0 \rc +t\lc F^1 \rc +\cdots + t^n\lc F^n \rc$ in $\fund{\CHx}^{-n}/2$. Applying the ring morphism $\fund{\CHx} \to \Zz$ given by $t \mapsto 1$ yields the statements, since by \rref{lemm:chi_odd} the image of $\fund{\CHx}^{-n}$ is contained in $2\Zz$ when $n$ is odd.
\end{proof}

\begin{lemma}
\label{lemm:nonzerosection}
Assume that $k$ is infinite. Let $S$ be a quasi-projective $k$-scheme and $V \to S$ a vector bundle of rank $r > \dim S$. Then for every $s \in \Zz-\{0\}$, we may find a line bundle $L$ fitting into an exact sequence of vector bundles over $S$
\[
0 \to W \to V \to L^{\otimes s} \to 0.
\]
\end{lemma}
\begin{proof}
Let $A$ be an ample line bundle over $S$. Then for any large enough integer $m$, the vector bundle $G=V^\vee\otimes A^{\otimes m}$ is generated by its global sections. Fix such an $m$ divisible by $s$, and a finite dimensional $k$-vector space $\Sigma \subset H^0(X,G)$ generating $G$. The kernel $K$ of the surjective morphism of vector bundles $S \times \Sigma \to G$ is a vector bundle whose rank is $\dim \Sigma - r$. Then $\dim K = \dim S + \dim \Sigma -r < \dim \Sigma$, hence the composite $K \to S \times \Sigma \to \Sigma$ is not dominant. Thus there is a nonempty open subscheme $U$ of $\Sigma$ such that $K \cap (S \times U) = \varnothing$. It follows that every $k$-rational point of $U$ is a nowhere vanishing section of $G$. Since $U$ is a nonempty open subscheme of an affine space over the infinite field $k$, it admits a $k$-rational point, so that $G$ admits a nowhere vanishing section. This gives a surjective morphism of vector bundles $V \to A^{\otimes m}$, so we let $W$ be its kernel and $L=A^{\otimes (m/s)}$.
\end{proof}

\begin{lemma}
\label{lemm:CHx_rank_dim}
Let $S$ be a smooth projective $k$-scheme of pure dimension $d$ and $V \to S$ be a vector bundle of rank $r$. Then for any $0 \leq m \leq r -d$, we have
\[
\lc c_1^{\CHx}(\Oc_{\Pp(V)}(1))^m \rc = (r-m)t^{r-1-m}\lc S \rc \in \fund{\CHx} \subset \Zz[t].
\]
\end{lemma}
\begin{proof}
We may assume that $k$ is infinite. We proceed by induction on $m$, the case $m=0$ being \rref{lemm:chi_P}. Let $m>0$, and $s \in \Zz-\{0\}$. By \rref{lemm:nonzerosection} there is an exact sequence $0 \to W \to V \to L^{\otimes s} \to 0$ with $L$ a line bundle. By \rref{lemm:CHx_FGL}, we have $c_1^{\CHx}(L^{\otimes s}) \in s \CHx(S)$. The line bundle of the effective Cartier divisor $\Pp(W) \to \Pp(V)$ is $p^*L^{\otimes s}(1)$, where $p\colon \Pp(V) \to S$ is the projective bundle, hence in view of \rref{p:c1_Cartier}
\[
[\Pp(W)] = c_1^{\CHx}(\Oc_{\Pp(V)}(1)) +_{\CHx} c_1^{\CHx}(p^*L^{\otimes s}) = c_1^{\CHx}(\Oc_{\Pp(V)}(1)) \in \CHx(\Pp(V))/s.
\]
Multiplying with $c_1(\Oc_{\Pp(V)}(1))^{m-1}$ and projecting to $\Spec k$, we obtain
\[
\lc c_1^{\CHx}(\Oc_{\Pp(W)}(1))^{m-1} \rc = \lc c_1^{\CHx}(\Oc_{\Pp(V)}(1))^m \rc \in \CHx(\Spec k)/s = \Zz[t]/s.
\]
Using the induction hypothesis on the bundle $W$, we deduce that $\lc c_1^{\CHx}(\Oc_{\Pp(V)}(1))^m \rc - (r-m)t^{r-1-m}\lc X \rc \in s\Zz[t]$. This element is divisible in $\Zz[t]$, hence vanishes.
\end{proof}

\begin{theorem}
\label{th:E_4}
Let $X$ be a smooth projective $k$-scheme of pure dimension $n$ with a $\mud$-action. If $2 \dim X^\mud < n-1$, then $\chi(X^\mud)$ is divisible by four.
\end{theorem}
\begin{proof}
Let $N$ be the normal bundle to $X^\mud \to X$ and write $\xi = c_1(\Oc_{\Pp(N\oplus 1)}(1))$ in $\CHx(\Pp(N\oplus 1))$. By \rref{th:Lmod2} and \rref{lemm:CHx_FGL}, in view of \eqref{p:2Cc} we have in $\fund{\CHx}/2$ 
\begin{equation}
\label{eq:xi_xid}
0=\Big\lc 2\frac{[-1]_{\CHx}(\xi)^2}{[-2]_{\CHx}(\xi)} \Big\rc = \Big\lc 2\frac{(-\xi)^2(1-2t\xi)^{-2}}{-2\xi(1-3t\xi)^{-1}} \Big\rc =-\lc \xi \rc - t\lc \xi^2 \rc,
\end{equation}
\begin{equation}
\label{eq:xi_xid2}
0=\Big\lc 2\frac{[-1]_{\CHx}(\xi)^3}{[-2]_{\CHx}(\xi)} \Big\rc = \Big\lc 2\frac{(-\xi)^3(1-2t\xi)^{-3}}{-2\xi(1-3t\xi)^{-1}} \Big\rc =\lc \xi^2 \rc - t\lc \xi^3 \rc.
\end{equation}
Let $d=\dim X^\mud$ and write $X^\mud$ as the disjoint union of the schemes $F^{n-d},\cdots,F^n$, where $F^r$ has pure dimension $n-r$. Since $n>2d$ by assumption, for every $r=n-d,\cdots,n$, the vector bundle $N^r=N|_{F^r}$ has rank $r>  \dim F^r$. Applying \rref{lemm:CHx_rank_dim} for $m=1$ and $m=2$ to each bundle $N^r \oplus 1$ yields in $\fund{\CHx}^{1-n}/2$
\[
\lc \xi \rc - t\lc \xi^2 \rc = \sum_{r=n-d}^n rt^{r-1} \lc F^r \rc - \sum_{r=n-d}^n (r-1)t^{r-1} \lc F^r \rc = \sum_{r=n-d}^n t^{r-1} \lc F^r \rc.
\]
This element vanishes in $\fund{\CHx}^{1-n}/2$ by \eqref{eq:xi_xid}. If $n$ is even, the ring morphism $\fund{\CHx} \to \Zz$ given by $t \mapsto 1$ maps $\fund{\CHx}^{1-n}$ to $2 \Zz$ by \rref{lemm:chi_odd}, hence $\chi(X^\mud) = \chi(F^{n-d}) +\cdots +\chi(F^n) \in 4\Zz$, concluding the proof in that case.

Now assume that $n$ is odd. Since $2d <n-1$, each $N^r$ has rank $r> \dim F^r +1$. Applying \rref{lemm:CHx_rank_dim} for $m=2$ and $m=3$ to each bundle $N^r\oplus 1$ yields in $\fund{\CHx}^{2-n}/2$
\[
\lc \xi^2 \rc - t\lc \xi^3 \rc = \sum_{r=n-d}^n (r-1)t^{r-2} \lc F^r \rc - \sum_{r=n-d}^n (r-2)t^{r-2} \lc F^r \rc = \sum_{r=n-d}^n t^{r-2} \lc F^r \rc.
\]
This element vanishes in $\fund{\CHx}^{2-n}/2$ by \eqref{eq:xi_xid2}. The ring morphism $\fund{\CHx} \to \Zz$ given by $t \mapsto 1$ maps $\fund{\CHx}^{2-n}$ to $2 \Zz$ by \rref{lemm:chi_odd}, hence $\chi(X^\mud) = \chi(F^{n-d}) +\cdots +\chi(F^n) \in 4\Zz$.
\end{proof}

\begin{corollary}
\label{cor:E_dim}
Let $X$ be a smooth projective $k$-scheme of pure dimension $n$ with a $\mud$-action.
\begin{enumerate}[(i)]
\item \label{cor:E_dim:1} If $\chi(X)$ is odd, then $2 \dim X^\mud \geq n$.

\item \label{cor:E_dim:2} If $n$ is odd and $\chi(X)$ is not divisible by four, then $2 \dim X^\mud +1 \geq n$.
\end{enumerate}
\end{corollary}
\begin{proof}
Combine \rref{prop:E_fix} with \rref{th:E_4}. Note that $n$ must be even if $\chi(X)$ is odd by \rref{lemm:chi_odd}, so that $2 \dim X^\mud +1 \geq n$ implies that $2 \dim X^\mud \geq n$.
\end{proof}

\section{Additive Chern number}
\label{sect:indecomposable}

\subsection{The theory {\texorpdfstring{$\CHa$}{A}}}

\begin{para}
We will denote by $(0)$ the empty partition. Let $n\in \Nn$. Taking the partition $\alpha=(n)$ and the theory $\Hh=\CH$ in \rref{p:def_CF}, we have a Conner--Floyd Chern class $c_{(n)}(E) \in \CH(X)$ for all $X \in \Sm_k$ and $E\in K_0(X)$. Observe that $c_{(0)}(E)=1$. If $n>0$, then $c_{(n)}(E+F) = c_{(n)}(E) + c_{(n)}(F)$ for every $E,F \in K_0(X)$, and $c_{(n)}(L) = c_1(L)^n$ for every line bundle $L \to X$. If $X$ is a smooth projective $k$-scheme of pure dimension $n$, its \emph{additive Chern number} is the integer
\begin{equation}
\label{eq:cn}
c_{(n)}(X) =  \deg (c_{(n)}(-\Tan_X)) \in \Zz.
\end{equation}
\end{para}

\begin{definition}
We consider the functor $\CHa = \CHt \otimes_{\Zz[\bb]} \Zz[t,\varepsilon]/\varepsilon^2$ where $\Zz[\bb] \to \Zz[t,\varepsilon]/\varepsilon^2$ is the morphism $b_i \mapsto \varepsilon t^i$. The element $t$ has degree $-1$ and $\varepsilon$ has degree zero. It follows from \rref{prop:tilde_is_oct}, \rref{ex:tensor_oct} and \rref{p:CH_is_oct} that $\CHa$ is an oriented cohomology theory. For any $X\in \Sm_k$, we have $\CHa(X) = \CH(X)[t,\varepsilon]/\varepsilon^2$, and for any morphism $f\colon Y \to X$ in $\Sm_k$ we have $f^*_{\CHa}=f^*_{\CH}$. If $f$ is projective with virtual tangent bundle $\Tan_f \in K_0(Y)$, then for any $a \in \CHa(Y)$,
\begin{equation}
\label{eq:pf_CHa}
f_*^{\CHa}(a) = f_*^{\CH} \Big( \big(1-\varepsilon \sum_{i \geq 1} c_{(i)}(\Tan_f) t^i\big) a \Big) \in \CHa(X).
\end{equation}
\end{definition}

\begin{para}
\label{p:c1_CHa}
Let $L$ be a line bundle over $X\in \Sm_k$. By \rref{p:c1_tilde} we have
\[
c_1^{\CHa}(L) = c_1^{\CH}(L) + \varepsilon \sum_{i \geq 1} t^i c_1^{\CH}(L)^{i+1} \in \CHa(X).
\]
In particular $\varepsilon c_1^{\CHa}(L)^j = \varepsilon c_1^{\CH}(L)^j \in \CHa(X)$ for any $j \in \Nn$.
\end{para}

\begin{lemma}
\label{lemm:FGL_A}
The formal group law of the theory $\CHa$ is given by
\[
x+_{\CHa} y = x + y +\varepsilon \sum_{i \geq 1} \Big( (x+y)^{i+1} - x^{i+1} -y^{i+1}\Big) t^i \in (\Zz[t,\varepsilon]/\varepsilon^2)[[x,y]].
\]
\end{lemma}
\begin{proof}
Let $L,M$ be line bundles over $X \in \Sm_k$. Write $l=c_1^{\CH}(L),m=c_1^{\CH}(M)$ and $\lambda = c_1^{\CHa}(L), \mu = c_1^{\CHa}(M)$ in $\CHa(X)$. By \rref{p:additive} and \rref{p:c1_CHa}, we have in $\CHa(X)$
\[
c_1^{\CHa}(L\otimes M) = l+m+\varepsilon\sum_{i\geq 1} t^i (l+m)^{i+1} = \lambda + \mu +  \varepsilon\sum_{i\geq 1} t^i \big((l+m)^{i+1} -l^{i+1} -m^{i+1}\big)
\]
and the statement follows from the last sentence of \rref{p:c1_CHa}.
\end{proof}

\begin{para}
\label{eq:A_posdim}
When $X$ is a smooth projective $k$-scheme of pure dimension $n$, we have in $\fund{\CHa} \subset \CHa(\Spec k) = \Zz[t,\varepsilon]/\varepsilon^2$
\[
\lc X\rc =\begin{cases}
c_{(n)}(X) &\text{ if $n=0$,}\\
c_{(n)}(X) \varepsilon t^n &\text{ if $n>0$.}
\end{cases}
\]
\end{para}

\begin{lemma}
\label{lemm:Acn}
We have
\[
\fund{\CHa}^{-n}=
\begin{cases}
0 & \text{ if $n<0$,} \\
\Zz &\text{ if $n=0$,} \\
p \varepsilon t^n\Zz &\text{ if $n=p^q-1$ for some prime $p$ and integer $q\geq 1$,} \\
\varepsilon t^n\Zz&\text{ otherwise.}
\end{cases}
\]
\end{lemma}
\begin{proof}
We may assume that $n>0$. By \rref{cor:Lazard}, the ring $\fund{\CHa}$ is generated by the coefficients $a_{i,j}$ of $F_{\CHa}(x,y)$. Now \rref{eq:A_posdim} implies that $\fund{\CHa}^r \cdot \fund{\CHa}^s =0$ when $r,s \in \Zz -\{0\}$, hence the group $\fund{\CHa}^{-n}$ is generated by the coefficients $a_{i,n+1-i}$. Now \rref{lemm:FGL_A} implies that $a_{i,n+1-i} = \binom{n+1}{i} \varepsilon t^n$ when $0<i<n+1$, and $a_{i,n+1-i}=0$ otherwise. Thus the statement follows from the computation
\[
\gcd_{0 < i<n+1} \binom{n+1}{i} = 
\begin{cases}
p &\text{ if $n=p^q-1$ for some prime $p$ and integer $q\geq 1$,} \\
1&\text{ otherwise.}
\end{cases}\qedhere
\]
\end{proof}

\subsection{Additive Chern number and fixed locus}

\begin{lemma}
\label{lemm:X_PN_A}
Let $X$ be a smooth projective $k$-scheme of pure dimension $n$ with a $\mud$-action. Let $N$ be the normal bundle to the immersion $X^\mud \to X$.
\begin{enumerate}[(i)]
\item \label{lemm:X_PN_A:1} We have $c_{(n)}(X) = c_{(n)}(\Pp(N\oplus 1)) \in \Fd$, and if $0 < j \leq n$
\[
\deg (c_1(\Oc_{\Pp(N \oplus 1)}(1))^jc_{(n-j)}(\Tan_{\Pp(N\oplus 1)})) =0 \in \Fd.
\]

\item \label{lemm:X_PN_A:2} Let $0< j\leq n$ and assume that $n+1$ and $n-j+1$ are powers of two. Then
\[
c_{(n)}(X) = c_{(n)}(\Pp(N \oplus 1)) + \deg (c_1(\Oc_{\Pp(N \oplus 1)}(1))^j c_{(n-j)}(\Tan_{\Pp(N\oplus 1)})) \in \Zz/4.
\]
\end{enumerate}
\end{lemma}
\begin{proof}
Let $z = c_1^{\CH}(\Oc_{\Pp(N \oplus 1)}(-1))$. Since $z=-c_1^{\CH}(\Oc_{\Pp(N \oplus 1)}(1))$ by \rref{p:additive}, we may replace $\Oc_{\Pp(N \oplus 1)}(-1)$ by $\Oc_{\Pp(N \oplus 1)}(1)$ in \eqref{lemm:X_PN_A:1}, and thus also in \eqref{lemm:X_PN_A:2}. Let $\zeta = c_1^{\CHa}(\Oc_{\Pp(N \oplus 1)}(-1))$. Let $0 \leq m \leq n$. Then by \rref{p:c1_CHa} and \eqref{eq:pf_CHa} we have in $\fund{\CHa}^{m-n}$
\[
\lc \zeta^m \rc = \deg \Big( \big(z + \varepsilon \sum_{i \geq 1} t^i z^{i+1}\big)^m \big(1 - \varepsilon \sum_{i \geq 1} t^i c_{(i)}(\Tan_{\Pp(N\oplus 1)})\big)\Big)
\]
hence in $\fund{\CHa}^{m-n}$
\begin{equation}
\label{eq:zetaj_CHa}
\lc \zeta^m \rc = 
\begin{cases}
\deg (z^n)  &\text{ if $m=n$,}\\
\varepsilon t^{n-m}\big(m\deg (z^n) - \deg (z^mc_{(n-m)}(\Tan_{\Pp(N\oplus 1)}))\big) &\text{ if $m<n$.}
\end{cases}
\end{equation}
Now \rref{lemm:FGL_A} implies that we have in $\fund{\CHa}^{m-n}$ (note that $\varepsilon \lc \zeta^i \rc=0$ when $i\neq n$)
\[
\Big\lc \frac{2\zeta^{m+1}}{[2]_{\CHa}(\zeta)} \Big\rc  = \Big\lc \zeta^m\Big(1-\varepsilon\sum_{i\geq 1} t^i(2^i-1)\zeta^i\Big)\Big\rc = \lc \zeta^m \rc + \varepsilon t^{n-m}(1-2^{n-m})\lc \zeta^n \rc.
\]
We combine this equation with \eqref{eq:zetaj_CHa}, and apply \rref{th:Lmod2}. In case $m=n=0$, we obtain $\lc X \rc = \deg(z^0)$ in $\fund{\CHa}^0/2 = \Zz/2$, proving the lemma when $n=0$. Thus we assume that $n>0$ from now on. In case $m=n$, we obtain $\deg(z^n) \in 2\fund{\CHa}^0 = 2\Zz$, while for $m<n$ we obtain in $\fund{\CHa}^{m-n}/2$
\[
\varepsilon t^{n-m} \big((m+1- 2^{n-m}) \deg (z^n)-\deg (z^mc_{(n-m)}(\Tan_{\Pp(N\oplus 1)}))\big) =
\begin{cases}
\lc X \rc &\text{ if $m=0$,} \\
0 &\text{ if $0<m<n$.}
\end{cases}
\]
Taking \rref{lemm:Acn} into account, we apply the group morphism $\fund{\CHa} \to \Zz$ defined by $1\mapsto 1$ and $\varepsilon t^s \mapsto 1$ if $s>0$. Then \eqref{lemm:X_PN_A:1} follows by letting $m=0$ and $m=j$. We now prove \eqref{lemm:X_PN_A:2}. We have, letting $m=0$,
\begin{equation}
\label{eq:cn_zn}
c_{(n)}(X) = \deg (z^n) + c_{(n)}(\Pp(N\oplus 1)) \in \Zz/4.
\end{equation}
This proves \eqref{lemm:X_PN_A:2} in case $j=n$. Finally, assuming that $j<n$,  observe that the integer $(j + 1- 2^{n-j})$ is odd, hence letting $m=j$,
\[
\deg (z^n)=\deg \big(z^jc_{(n-j)}(\Tan_{\Pp(N\oplus 1)})\big)\in \Zz/4.
\]
Combining this equation with \eqref{eq:cn_zn} yields \eqref{lemm:X_PN_A:2} in case $j<n$.
\end{proof}

\begin{theorem}
\label{th:additive_Chern}
Let $X$ be a smooth projective $k$-scheme of pure dimension $n$ with a $\mud$-action. Assume that $2 \dim X^\mud < n-1$. Then $c_{(n)}(X)$ is even. If $n=2^q-1$ for some $q \geq 1$, then $c_{(n)}(X)$ is divisible by four.
\end{theorem}
\begin{proof}
If $n\in \{0,1\}$, then $X^{\mud} = \varnothing$, hence $\Pp(N\oplus 1) = \varnothing$. Thus the theorem follows from  \dref{lemm:X_PN_A}{lemm:X_PN_A:1} when $n=0$, and from \dref{lemm:X_PN_A}{lemm:X_PN_A:2} with $j=1$ when $n=1$.

We now assume that $n\geq 2$. Let $k$ be the integer such that $n=2k$ or $n=2k+1$, and set $l=n-k$. Let $p\colon \Pp(N\oplus 1) \to X^\mud$ be the projective bundle, and write $\xi = c_1(\Oc(1)) \in \CH(\Pp(N \oplus 1))$. Since $\Tan_{\Pp(N\oplus 1)}=(p^*N\oplus 1)(1) + p^*\Tan_{X^\mud} -1 \in K_0(\Pp(N\oplus 1))$ (see e.g.\ \cite[\S B.5.8]{Ful-In-98}), for any $i=0,\cdots,l$ we have $c_{(n-i)}(\Tan_{\Pp(N\oplus 1)}) = c_{(n-i)}((p^*N\oplus 1)(1))+ p^*c_{(n-i)}(\Tan_{X^\mud})$ in $\CH(\Pp(N\oplus 1))$ (note that $n-i>0$). For such $i$, the element $c_{(n-i)}(\Tan_{X^\mud}) \in \CH^{n-i}(X^\mud)$ vanishes, since $\dim X^\mud < k \leq n-i$. Thus for $i=0,\cdots,l$ we have in $\CH(\Pp(N\oplus 1))$,
\begin{equation}
\label{eq:TP_N1}
c_{(n-i)}(\Tan_{\Pp(N\oplus 1)}) = c_{(n-i)}((p^*N\oplus 1)(1)).
\end{equation}

Now, we have in $\Zz[x,y]$
\[
(x+y)^kx^l = (x+y)^k((x+y) -y)^l = (x+y)^n + \sum_{i=1}^l (-1)^i\binom{l}{i} y^i(x+y)^{n-i}. 
\]
Expanding the factor $(x+y)^k$ on the left hand side yields in $\Zz[x,y]$
\[
(x+y)^n = \sum_{j=0}^k \binom{k}{j} y^{k-j}x^{l+j}  - \sum_{i=1}^l (-1)^i\binom{l}{i} y^i(x+y)^{n-i}.
\]
From the splitting principle \rref{p:splitting}, we deduce, in $\CH(\Pp(N\oplus 1))$
\[
c_{(n)}((p^*N\oplus 1)(1)) = \sum_{j=0}^k \binom{k}{j} \xi^{k-j} p^*c_{(l+j)}(N\oplus 1) - \sum_{i=1}^l (-1)^i\binom{l}{i} \xi^ic_{(n-i)}((p^*N\oplus 1)(1)).
\]
Now $c_{(l+j)}(N\oplus 1) \in \CH^{l+j}(X^\mud)$ vanishes for any $j\geq 0$, because $\dim X^\mud < k \leq l +j$. In view of \eqref{eq:TP_N1}, we obtain in $\CH(\Pp(N\oplus 1))$
\begin{equation}
\label{eq:cn_TPN}
-c_{(n)}(\Tan_{\Pp(N\oplus 1)}) = \sum_{i=1}^l (-1)^i\binom{l}{i} \xi^ic_{(n-i)}(\Tan_{\Pp(N\oplus 1)}).
\end{equation}
Applying \dref{lemm:X_PN_A}{lemm:X_PN_A:1} and taking the degree of \eqref{eq:cn_TPN} yields $c_{(n)}(X)=0\in\Fd$. Now assume that $n=2^q-1$ with $q\geq 2$. Then $l=2^{q-1}$, and we have
\[
\binom{l}{i} \mod 4= 
\begin{cases}
1 &\text{ if $i=0$ or $i=l$,} \\
2 &\text{ if $i=2^{q-2}$,}\\
0 &\text{ otherwise.}
\end{cases}
\]
Thus, taking the degree of \eqref{eq:cn_TPN} yields
\[
c_{(n)}(\Pp(N\oplus 1)) = \deg(\xi^l c_{(n-l)}(\Tan_{\Pp(N\oplus 1)}))
+ 2 \deg (\xi^{2^{q-2}} c_{(n-2^{q-2})}(\Tan_{\Pp(N\oplus 1)})) \in  \Zz/4.
\]
Applying \dref{lemm:X_PN_A}{lemm:X_PN_A:2} with $j=l$, we get
\[
c_{(n)}(X) = 2\deg(\xi^l c_{(n-l)}(\Tan_{\Pp(N\oplus 1)})) + 2 \deg (\xi^{2^{q-2}} c_{(n-2^{q-2})}(\Tan_{\Pp(N\oplus 1)})) \in \Zz/4,
\]
which vanishes by \dref{lemm:X_PN_A}{lemm:X_PN_A:1} applied with $j=l$ and $j=2^{q-2}$.
\end{proof}

\subsection{Decomposability in the Lazard ring}
\begin{para}
Let $R$ be a $\Zz$-graded ring. Denote by $I_R$ the ideal generated by homogeneous elements of nonzero degrees in $R$, and set $\Dec{R} = (I_R)^2$. An element of $R$ is called \emph{decomposable} if it belongs to $\Dec{R}$, and \emph{indecomposable} otherwise.
\end{para}

Decomposability in $\Laz/p$ or $\Lp$ is detected using the additive Chern number:
\begin{lemma}
\label{lemm:indec_Ld_L2}
Let $X$ be a smooth projective $k$-scheme of pure dimension $n$, and $p$ a prime number.
\begin{enumerate}[(i)]
\item \label{lemm:indec_Ld_L2:1} If $n=p^q -1$ for some $q \geq 1$, then $\lc X\rc$ is decomposable in $\Lp$.

\item \label{lemm:indec_Ld_L2:2} Assume that $n \neq p^q-1$ for all $q \geq 1$. Then $\lc X\rc$ is decomposable in $\Laz/p$ if and only if $\lc X\rc$ is decomposable in $\Lp$.

\item \label{lemm:indec_Ld_L2:3} The class $\lc X\rc$ is decomposable in $\Laz/p$ if and only if
\[
c_{(n)}(X) \in 
\begin{cases}
p^2\Zz &\text{ if $n =p^q-1$ for some $q \geq 1$,} \\
p\Zz &\text{ otherwise.}
\end{cases}
\]
\end{enumerate}
\end{lemma}
\begin{proof}
The degree zero components of $\Laz/p$ and $\Lp$ may be identified with $\Fp$, via the map $\lc X \rc \mapsto c_{(0)}(X)$. This implies the statements when $n=0$.

We now assume that $n>0$. In view of \rref{th:Lazard} and \rref{p:FGL_exp}, it follows from  \cite[II, \S7]{Adams-stable} that the subring $\Laz$ of $\Zz[\bb]$ is a polynomial ring in the variables $y_i$ for $i \in \Nn-\{0\}$, where $y_i$ is homogeneous of degree $-i$. In addition $y_i = v_i b_i \mod \Dec{\Zz[\bb]}$ with $v_i = l$ if $i = l^q -1$ for some prime $l$ and integer $q \geq 1$, and $v_i=1$ otherwise. Since the ring $\Laz$ is generated by the elements $y_i$, among which only $y_n$ has degree $-n$, we have $\lc X\rc = u y_n \mod \Dec{\Laz}$, for some $u \in \Zz$. Then $\lc X\rc$ is decomposable in $\Laz/p$ if and only if $u \in p\Zz$. Since  $\lc X \rc = uv_n b_n \mod \Dec{\Zz[\bb]}$, its $b_n$-coefficient $c_{(n)}(X)$ equals $uv_n$,  and \eqref{lemm:indec_Ld_L2:3} follows.

Assume that $n\neq p^q-1$ for all $q \geq 0$. Obviously if $\lc X \rc$ is decomposable in $\Laz/p$, then $\lc X \rc$ is decomposable in $\Laz_p$. Conversely assume that $\lc X\rc$ is decomposable in $\Laz_p \subset \Fp[\bb]$. Since $\lc X \rc = uv_n b_n \mod \Dec{\Fp[\bb]}$, it follows that $uv_n \in p\Zz$. Since $v_n$ is prime to $p$, this implies that $u \in p\Zz$, proving \eqref{lemm:indec_Ld_L2:2}.

Assume that $n=p^q-1$ for some $q \geq 1$. Since $v_n \in \Fp$ vanishes, so does the element $y_n = v_n b_n \in \Laz_p \subset \Fp[\bb]$. Since $\lc X\rc = u y_n \mod \Dec{\Laz_p}$, it follows that $\lc X\rc \in \Dec{\Laz_p}$, proving \eqref{lemm:indec_Ld_L2:1}.
\end{proof}

Besides the additive Chern number, other Chern numbers are affected by the decomposability in $\Laz/p$:

\begin{lemma}
\label{lemm:p-adic}
Let $p$ be a prime number, and $\alpha=(\alpha_1,\cdots,\alpha_m)$ a partition such that each $\alpha_i+1$ is a power of $p$. Let $X$ be a connected smooth projective $k$-scheme of positive dimension. Then $c_\alpha(X) \in p\Zz$. If $\lc X\rc \in \Laz/p$ is decomposable, then $c_\alpha(X) \in p^2\Zz$.
\end{lemma}
\begin{proof}
Let $\varphi$ be the ring endomorphism of $\Zz[\bb]$ mapping  $b_i$ to itself if $i+1$ is a power of $p$, and to $0$ otherwise. Then every element of $\Zz[\bb]$ has the same $b_{\alpha}$-coefficient as its image under $\varphi$. Let $\psi \colon \Zz[\bb] \to \Fp[\bb]$ be the composite of $\varphi$ with the reduction modulo $p$. The functor $\Hh=\CHt \otimes_{\Zz[\bb],\psi} \Fp[\bb]$ defines an oriented cohomology theory by \rref{prop:tilde_is_oct}, \rref{ex:tensor_oct} and \rref{p:CH_is_oct}. 
If $L$ is a line bundle over $X \in \Sm_k$ then by \rref{p:c1_tilde} 
\[
c_1^{\Hh}(L) = \sum_{i\in \Nn} c_1^{\CH}(L)^{i+1}\psi(b_i) = \sum_{q \in \Nn} c_1^{\CH}(L)^{p^q}b_{p^q-1} \in \Hh(X).
\]
Since $\CH$ is additive \rref{p:additive} and $p\Hh(X)=0$, it follows that $\Hh$ is additive. Therefore when $i+j>1$, the coefficient $a_{i,j} \in \Laz$ of $F_{\CHt}$ (defined in \eqref{eq:FGL}) lies in the kernel of $\psi$. By \rref{cor:Lazard}, this implies that $\varphi|_{\Laz} \colon \Laz \to \Zz[\bb]$ sends homogeneous elements of negative degrees to $p\Zz[\bb]$. Thus $\varphi(\lc X \rc) \in p\Zz[\bb]$, hence the $b_\alpha$-coefficient $c_\alpha(X)$ of $\lc X \rc \in \Zz[\bb]$ belongs to $p\Zz$. It also follows that $\varphi|_{\Laz}$ sends homogeneous elements of negative degrees whose image in $\Laz/p$ is decomposable to $p^2\Zz[\bb]$. Thus if $\lc X\rc \in \Laz/p$ is decomposable, then the $b_\alpha$-coefficient $c_\alpha(X)$ belongs to $p^2\Zz$.
\end{proof}

Finally \dref{lemm:indec_Ld_L2}{lemm:indec_Ld_L2:3} implies the following reformulation of \rref{th:additive_Chern}.
\begin{theorem}
\label{th:indec}
Let $X$ be a connected smooth projective $k$-scheme with a $\mud$-action. If $2 \dim X^\mud < \dim X-1$, then $\lc X\rc$ is decomposable in $\Laz/2$.
\end{theorem}

\end{document}